\documentclass[12pt]{amsart}
\usepackage{euler, amsfonts, amssymb, latexsym, epsfig, tikz-cd}%,epic,graphicx} %,mathabx}
\usetikzlibrary{arrows,automata,positioning,calc,decorations.pathmorphing,decorations.markings,decorations.pathreplacing,decorations.markings,decorations.pathreplacing,decorations.pathmorphing}
\usepackage[driverfallback=dvipdfm]{hyperref}
\usepackage{cancel,blkarray,romanbar}
\usepackage[enableskew]{youngtab}

\usepackage{cancel,manfnt,mathtools}
\usepackage[OT2,T1]{fontenc} 

\newcommand\junk[1]{}

\DeclareMathOperator{\coker}{coker}

\DeclareSymbolFont{cyrletters}{OT2}{wncyr}{m}{n} 

\DeclareMathSymbol{\Gam}{\mathalpha}{cyrletters}{"47}
\DeclareMathSymbol{\Ca}{\mathalpha}{cyrletters}{"57}

\newcommand\lie[1]{{\mathfrak #1}}

\newcommand\iso{\mathrel{\cong}}
\newcommand\CC{\mathbb C}
\newcommand\tensor{{\otimes}}

\newcommand\calO{{\mathcal O}}
\newcommand\calM{{\mathcal M}}

\newcommand\calL{{\mathcal L}}

\newtheorem{Theorem}{Theorem}
\newtheorem{Proposition}{Proposition}
\newtheorem{Lemma}{Lemma}

\newtheorem{Corollary}{Corollary}
\newtheorem*{Corollary*}{Corollary}
\newtheorem*{Lemma*}{Lemma}

\newtheorem*{Theorem*}{Theorem}
\theoremstyle{remark}
\newtheorem{Example}{Example}

\newcommand\onto{\mathrel{\twoheadrightarrow}}

\newcommand\into{\mathrel{\hookrightarrow}}

\newcommand\union{\cup}
\newcommand\Union{\bigcup}

\newcommand\PP{{\mathbb P}}

\newcommand\ZZ{{\mathbb Z}}
\newcommand\RR{{\mathbb R}}
\newcommand\rationals{{\mathbb Q}}

\newcommand\NN{{\mathbb N}}

\newcommand\Oplus{\bigoplus}

\theoremstyle{plain}

\theoremstyle{remark}

\renewenvironment{quotation}
{\list{}{%\setlength\listparindent{0.5em}%
    \setlength\itemindent{0em}%
    \setlength\leftmargin{1.5em}
    \setlength\rightmargin{1.5em}
  }%
\item[]}
{\endlist}

\newcommand\defn[1]{{\bf #1}} % maybe should be \em
\newcommand\actson{\circlearrowright}

\newcommand\Spec{{\rm Spec}\,}
\newcommand\Gm{{{\mathbb G}_m}}

\newcommand\QQ{\rationals}
\newcommand\FF{{\mathbb F}}

\newcommand\wt{\widetilde}

\font\co=lcircle10

\def\jr{\smash{\raise2pt\hbox{\co \rlap{\rlap{\char'005} \char'007}}
               \raise6pt\hbox{\rlap{\vrule height6.5pt}}
               \raise2pt\hbox{\rlap{\hskip4pt \vrule height0.4pt depth0pt
                width7.7pt}}}}
\def\je{\smash{\raise2pt\hbox{\co \rlap{\rlap{\char'005}
                \phantom{\char'007}}}\raise6pt\hbox{\rlap{\vrule height6pt}}}}
\def\+{\smash{\lower2pt\hbox{\rlap{\vrule height14pt}}
                \raise2pt\hbox{\rlap{\hskip-3pt \vrule height.4pt depth0pt
                width14.7pt}}}}
\def\textcross{\ \smash{\lower4pt\hbox{\rlap{\hskip4.15pt\vrule height14pt}}
                \raise2.8pt\hbox{\rlap{\hskip-3pt \vrule height.4pt depth0pt
                width14.7pt}}}\hskip12.7pt}
\def\textelbow{\ \hskip.1pt\smash{\raise2.8pt%
                \hbox{\co \hskip 4.15pt\rlap{\rlap{\char'005} \char'007}
                \lower6.8pt\rlap{\vrule height3.5pt}
                \raise3.6pt\rlap{\vrule height3.5pt}}
                \raise2.8pt\hbox{%
                  \rlap{\hskip-7.15pt \vrule height.4pt depth0pt width3.5pt}%
                  \rlap{\hskip4.05pt \vrule height.4pt depth0pt width3.5pt}}}
                \hskip8.7pt}

\usepackage{tikz}\usetikzlibrary{graphs,quotes,fit,positioning,matrix,calc,decorations.markings,angles,decorations.pathmorphing,decorations.pathreplacing}%,matrix,fit,shapes.geometric}

\tikzset{mynode/.style={circle,draw=black,fill=black,inner sep=1.8pt,outer sep=0pt}}
\tikzset{whitenode/.style={circle,draw=black,fill=white,inner sep=1.8pt,outer sep=0pt}}
\tikzset{edgelabel/.style={\mcol,inner sep=0pt}}
\tikzset{invlabel/.style={draw=black,text=black,circle,inner sep=0pt,minimum size=3mm}}
\tikzset{smallnode/.style={circle,draw=black,fill=black,
    inner sep=.01pt,outer sep=0pt}}

\newcommand\rhocek{{\check\rho}}
\newcommand\rhotilde{{\wt\rho}}

\begin{document}
\pagestyle{plain}

\title{Stable map quotients (and orbifold log resolutions)
  of Richardson varieties}

\author{Allen Knutson}
\email{allenk@math.cornell.edu}
\date{\today}

\maketitle

\renewcommand\AA{{\mathbb A}}

\begin{abstract}
  Let $X_\lambda^\mu := X_\lambda \cap X^\mu \subseteq G/P$ be a
  Richardson variety in a generalized partial flag manifold.  We use
  equivariant stable map spaces to define a canonical resolution
  $\wt{X_\lambda^\mu}$ of singularities, albeit obtaining an orbifold
  not a manifold.  The ``nodal curves'' boundary is an (orbifold)
  simple normal crossings divisor, and is conjecturally anticanonical.
  Its dual simplicial complex is the order complex of the open Bruhat
  interval $(\lambda,\mu) \subseteq W/W_P$, shown in [Bj\"orner-Wachs
  '82] to be a sphere or ball. In the case of $G/P$ a Grassmannian,
  the resolution $\wt{X_\lambda^\mu}$ is a GKM space, whose $T$-fixed
  points are indexed by rim-hook tableaux.
\end{abstract}

{\Small
  \setcounter{tocdepth}{3}
 \tableofcontents
}

\newcommand\eqdot{\dot =\, }

\section{Introduction}

Fix a pinning $(G,B,B_-,T,W)$ of a connected reductive complex algebraic group,
and a parabolic subgroup $P \leq G$, $P\geq B$.
For $\lambda,\mu \in W/W_P$, define the closed subvarieties of $G/P$
$$ X_\lambda := \overline{B_-\, \lambda P}/P,\quad
X^\mu := \overline{B \mu P}/P,\quad
X_\lambda^\mu := X_\lambda \cap X^\mu
$$
and call them a \defn{Schubert variety},
\defn{opposite Schubert variety}, and \defn{Richardson variety}
respectively. They are known to be normal with rational singularities.
One way (from \cite{Brion}; see also \cite{Escobar}) to resolve those
singularities involves choosing Bott-Samelson resolutions of 
$X_\lambda$ and $X^\mu$, whose fiber product is a {\em brick manifold} resolution.
In particular, this involves choosing reduced words for some elements of $W$.

In this paper we produce another resolution $\wt{X_\lambda^\mu}$,
with {\em no} choices involved, suffering the slight drawback that the
resolution is not by a manifold but by a smooth orbifold (or,
Deligne-Mumford stack). It is a couple of steps removed from a space
of genus zero stable maps, and we relate it in \S\ref{sec:Chow} to
Chow quotients.\footnote{Several years after we embarked on this study
  \cite{Allensdottirs} we learned that Mateusz Michalek, Leonid Monin,
  and Botong Wang were investigating the corresponding Chow quotients.
  In upcoming work,
  they make fascinating use of them to define a ``kernel'' in
  the poset $W/W_P$ (see \cite{FerroniEtAl} for the definition),
  where the palindromicity of the rational Betti numbers of
  $\wt{X_\lambda^\mu}$ plays a key role.}

If $P=B$, where we instead use Roman letters $u,v \in W$,
the \defn{boundary}
$$ \partial X_u^v := \Union_{u" \gtrdot u} X_{u'}^v \union
\Union_{v" \lessdot v} X_u^{v'}$$
(where $\gtrdot,\lessdot$ indicate covering relations)
is an anticanonical divisor \cite[lemma 5.4]{projRich}.
Some properties shared by both the brick manifold resolution and
$\wt{X_u^v}$ are that
\begin{itemize}
\item while not ``strict resolutions'' (ones that blow up no smooth points),
  they only blow up points on that boundary,
\item the preimage of $\partial X_u^v$ is a simple normal
  crossings divisor (albeit orbifold, in our case), and
\item its ``dual simplicial complex'' (in the sense of \cite{KollarXu,Payne};
  see \S\ref{sec:dualcomplex}) to that sncd is homeomorphic to a sphere.
\end{itemize}

The dual simplicial sphere to $\partial \wt{X_u^v}$ is the suspension of the
{\em order complex} of the open interval $(u,v) \subseteq W$ in
Bruhat order, shown combinatorially to be a sphere in \cite{BjornerWachs}.
Analogously, the dual simplicial sphere to the boundary of a brick manifold was
shown in  \cite{Escobar} to be a spherical subword complex.

Most of our results concern a simpler space we introduce, the ``stable map
quotient'' $X//_{StMap}\, S$ of a variety $X$ by the action of
a circle $S \iso \CC^\times$.
The resolution $\wt{X_\lambda^\mu}$ then arises via a ``graph space'' construction
$(X_\lambda^\mu \times \PP^1)//_{StMap}\, S_\Delta$ using the diagonal action,
where the action of $S$ on $X_\lambda^\mu$ comes from $S \leq T$ being the
first regular dominant coweight $\rhocek$.

After defining these spaces in \S\ref{sec:construction},
our first main result is the orbifold-smoothness of $\wt{X_\lambda^\mu}$.
% Using some standard machinery recapitulated in \S\ref{sec:weights},
We comment in \S\ref{sec:Chow} on the close relation of the stable map
quotient to the Chow quotient $X_\lambda^\mu //_{Chow}\, \rhocek$.
In \S\ref{sec:dualcomplex} we compute the stratification
of $X_\lambda^\mu//_{StMap}\, \rhocek$ induced from the sncd.
The resolution $\wt {X_\lambda^\mu}$ bears an action of the torus
$B_-\cap B = T$, and we study this action in \S\ref{sec:Tweights}.
While $\wt{X_\lambda^\mu}$ has isolated $T$-fixed
points and every point stratum is of course $T$-fixed, there are additional
$T$-fixed points making it tricky to compute Betti numbers,
which we do in the Grassmannian and $GL(n)/B$ cases.
In \S\ref{sec:2GKM} we further study the Grassmannian case,
in which $\wt{X_\lambda^\mu}$ has not only isolated fixed points but
isolated fixed curves (the ``Goresky-Kottwitz-MacPherson property'').
In \S\ref{sec:linebundles} we define some $T$-equivariant classes
on the $T$-fixed points, satisfying these GKM conditions, which seem likely
to be the first Chern classes of $T$-equivariant (orbi-)line bundles.
In \S\ref{sec:anticanonical} we address the sense in which the
boundary divisor $\partial \wt{X_\lambda^\mu}$ is anticanonical,
shedding new light on the difference between the ball and sphere cases
of \cite{BjornerWachs}.

\subsection{Acknowledgments}\label{ssec:ack}

I express my gratitude to Daniel Halpern-Leistner, Adeel Khan, Angelo
Vistoli, and Rachel Webb for helping me with the basics of stacks and
stable curves; any remaining mistakes, or even just gauchely
stated arguments, are mine.
I thank Christian Gaetz and Shiliang Gao for directing me to Dyer's result,
and Brian Hwang for directing me to Chen and Satriano's result.
Thanks also to Mateusz Michalek and Botong Wang for generously explaining
their interests in and applications of the Chow quotients, and to the
organizers of the Workshop on Combinatorics of Enumerative Geometry at
the Insitute for Advanced Study in February 2025 for serendipitously
bringing us together.

\section{Stable map quotients}\label{sec:construction}

Our reference for stable map spaces is \cite{FP}.

\subsection{Stable map spaces}

Given a curve class $\beta \in A_1(M)$ in the first Chow group
of a projective variety $M$, one defines the space $StMap_\beta(\PP^1,M)$
of \defn{genus zero stable maps} $\gamma\colon \Sigma \to M$
(with no ``marked points''),
where $\Sigma$ is an at-worst-nodal (and probably reducible)
connected curve of genus zero
(so, a tree of $\PP^1$s) and $\gamma_*([\Sigma]) = \beta$.
In addition, the group of automorphisms of $\gamma$ (maps $\Sigma\to\Sigma$
inducing a commuting triangle) is required to be finite;
equivalently, each component of $\Sigma$ that $\gamma$ collapses to
a point in $M$ should meet at least three other components of $\Sigma$.

So far this defines a {\em set} $\{[\Sigma,\gamma]\}$ of isomorphism
classes $[\Sigma,\gamma]$ of pairs $(\Sigma,\gamma)$; there is
more work involved to define a moduli {\em space}, and the finite groups of
automorphisms mean this space will generally be a {\em Deligne-Mumford stack}
or ``orbifold'' in the differential-topological language.

\begin{Theorem}\cite[part of Theorem 2]{FP}
  If $M$ is a generalized flag variety $G/P$ (or more generally is ``convex''),
  then $StMap_\beta(\PP^1,M)$ can be given the structure of a smooth orbifold.%
  \footnote{To be precise, in \cite{FP} they punt discussion of stacks.
    In their \S4 they construct their moduli spaces by gluing
    together quotients of manifolds by finite groups, leaving
    the reader to check that the stabilizer groups match across
    the gluing.}
\end{Theorem}

At each node $p$ of a genus zero stable curve $\Sigma$, the two components $C,D$
of $\Sigma$ that meet at $p$ are distinct (i.e. unlike on a nodal elliptic curve).
We relate the two tangent lines $T_p C$, $T_p D$ in the next subsection.

\subsection{Torus actions}\label{ssec:torusactions}

When $T \actson M$, we get an action $T \actson StMap_\beta(\PP^1,M)$ on
each stable map space,\footnote{To be more stackily precise
  as in \cite[A.6]{Adeel2}, this $T$-action
  is on the coarse moduli space of the stack of stable maps, and it is in this
  sense that we refer to ``fixed points''. One could consider the
  proposition that follows as constructing a group $\wt T$ that acts on
  the {\em stack} of stable maps not just on its coarse moduli space.
\junk{  in which case what we call ``$T$-fixed points'' are not necessarily
  $\wt T$-fixed.}}
defined by $t\cdot [\Sigma, \gamma] := [\Sigma, t \circ \gamma]$.
For a point to be fixed by some $t$, we need an isomorphism
$\phi_t\colon \Sigma\to\Sigma$ fitting into a commuting triangle

% https://q.uiver.app/#q=WzAsMyxbMCwwLCJcXFNpZ21hIl0sWzAsMSwiXFxTaWdtYSJdLFsxLDEsIk0iXSxbMCwyLCJcXGdhbW1hIl0sWzEsMiwiXFxnYW1tYSIsMl0sWzAsMSwiXFx2YXJwaGlfdCIsMl1d
\[\begin{tikzcd}
    \Sigma \\
    \Sigma & M
    \arrow["{\varphi_t}"', from=1-1, to=2-1]
    \arrow["\gamma", from=1-1, to=2-2]
    \arrow["t\,\circ\,\gamma"', from=2-1, to=2-2]
  \end{tikzcd}\]
By the stability assumption, there are only finitely many choices of
such $\phi_t$.

\begin{Proposition}\label{prop:extension}
  The $T$-fixed points $[\Sigma,\gamma]$ in $StMap_\beta(\PP^1,M)$
  correspond to curves $\Sigma$ bearing an action of a finite extension
  torus $\wt T \onto T$, such that the map $\gamma$ is $\wt T$-equivariant.
  Indeed, there is a single extension $\wt T\onto T$ that suffices for all
  the $T$-fixed points in $StMap_\beta(\PP^1,M)$. 
\end{Proposition}

To see an example where this extension is necessary, consider
$T = \Gm$ acting on $\PP^1$ in the usual way, and $\gamma\colon \PP^1 \to \PP^1$
being the $z\mapsto z^6$ map. There is no way to lift the action without
first extending the group (by at least $\ZZ/6$).

\begin{proof}
  For each point $[\Sigma,\gamma]$ ($T$-fixed or not), the set
  $\{(t,\varphi_t)\colon t\in T$, $\varphi_t$ fits in the commuting
  triangle above$\}$ forms a group $G$ acting on $\Sigma$.
  The projection $G \to T$ is onto exactly if $[\Sigma,\gamma]$
  is $T$-fixed, in which case it is a finite extension of $T$
  (by the stability assumption).
  
  Assume that now, and let $\wt T$ be the identity component of $G$.
  This is a finite connected extension of $T$, hence (at least in
  characteristic $0$) is a torus itself.

  For the second claim, observe that the construction of
  $StMap_\beta(\PP^1,M)$ in \cite[\S 4]{FP} involves gluing finitely
  many quotients $U/\Gamma$ together, each a smooth quasiprojective
  variety divided by a finite group. The stabilizers (which become
  the automorphism groups in the stack) define a finite stratification
  of each $U$. So there are finitely many strata to consider, on
  each of which there is a single finite extension. Take $\wt T$ to be
  the identity component of the fiber product of all of those extensions.
\end{proof}

In the rest of the paper we will put tildes on our tori to designate
one of these extensions; in particular, be on the lookout for $\rhotilde$
a finite extension of $\rhocek$.

If $H \trianglelefteq G$ is a normal subgroup and $G \actson M$,
then $G$ will preserve and therefore act on $StMap_\beta(\PP^1,M)^H$.
If $G$ is connected, then it will separately preserve each component
of $StMap_\beta(\PP^1,M)^H$. We will use these facts later, where $G$
will be a torus and hence every subgroup $H\leq G$ will be normal.

\junk{
Here's an example showing why the finite extensions are important
to consider. Let $M = StMap_{2[\PP^1]}(\PP^1,\PP^2)$ be the space of
degree $2$ stable maps to $\PP^2$, and $\gamma_f\colon \PP^1 \to \PP^2$
take $[a,b] \mapsto [a^2, fab, b^2]$. Let $T$ be the $1$-torus
acting on $\PP^2$ by $t \cdot [x,y,z] := [x, ty, t^2z]$.
Then for $f\neq 0$ we can lift the $T$-action to $\PP^1$ by
$t \cdot [a,b] := [a,tb]$, making the diagram
$$
\begin{matrix}
  [a, b]      &\mapsto& [a, tb]     \\
  \downmapsto &       & \downmapsto \\
  [a^2, 
\end{matrix}
$$
}
\begin{Proposition}\label{prop:dual}
  Let $[\Sigma,\gamma] \in StMap_\beta(\PP^1,M)^T$ be an \defn{equivariantly
    smoothable} stable curve, i.e. some stable map $[\Sigma',\gamma']$
  in a geometric component of $StMap_\beta(\PP^1,M)^T$ containing
  $[\Sigma',\gamma']$ is irreducible.
  Then at each node $p$ of $\Sigma$, the $\wt T$-weights on the tangent
  spaces $T_p C$, $T_p D$ to the two components containing $p$
  are dual as $\wt T$-representations. (Here $\wt T \actson \Sigma$ is the
  extension of $T$ guaranteed by proposition \ref{prop:extension}.)
\end{Proposition}

\junk{
\begin{proof}
  Pick the germ $\alpha\colon S\to StMap_\beta(\PP^1,M)^T$ of a smooth
  curve $S$ (the $\Spec$ of a discrete valuation ring), such that
  $S$'s closed point $s_0$ goes through $[\Sigma,\gamma]$ and $S$'s
  generic point maps to the generic point of $StMap_\beta(\PP^1,M)^T$.
  (This is a standard trick -- first pass to an affine neighborhood
  of $[\Sigma,\gamma]$, then cut down to something $1$-dimensional,
  take its normalization to smooth it, and finally take a completion.)

  Pull back the universal family on $StMap_\beta(\PP^1,M)$ to a family
  over $S$ of stable curves. In the closed fiber over $s_0$ we have
  the node $p$. Shrink the family to a formal neighborhood $F$
  of that node; now it is a formal smoothing $F \to S$ of the node.
  We use the Stacks Project lemma 53.20.11 to put formal local
  co\"ordinates $(x,y) \mapsto xy$ on the map $F \to S$.

  At this stage $\wt T$ acts on $F$ and acts trivially on $S$, and the
  map $F \to S$ is $\wt T$-invariant. Since the map is $\wt T$-invariant,
  the weights of $x,y$ (which show up on the special fiber $xy=0$)
  must be negative one another.
\end{proof}
}

\begin{proof}
  This is proven for finite groups in \cite[Remark 2.1.4, Lemma
  2.3.5]{WebbEtAl}. From there use the fact that union of the finite
  subgroups of $T$ is Zariski dense. (A more direct reason will
  show up in \S\ref{ssec:analyzing}(2).)
\end{proof}

\begin{Corollary}\label{cor:chain}
  If $\gamma \colon \Sigma \to M$ is a $\wt T$-equivariantly smoothable
  stable map of degree $\beta\neq 0$, and $M^{\wt T}$ is finite,
  then $\Sigma$ is a {\em chain} of $\PP^1$s (not just a tree), and
  all components are $\wt T$-isomorphic, each with two $\wt T$-fixed points
  and isotropy weights $+\alpha,-\alpha$ for some fixed $\alpha$.
  Also, no component of $\Sigma$ is collapsed to a point by $\gamma$.
\end{Corollary}

Hereafter we'll call such a chain of $\PP^1$s, with $\wt T$
acting the same way on each component, a \defn{caterpillar}.

\begin{proof}
  If $\wt T$ acts trivially on some component of $\Sigma$, then by the
  proposition it acts trivially on all neighboring components (meaning,
  those it intersects), hence by connectedness acts trivially
  on all of $\Sigma$. By the $\wt T$-equivariance of $\gamma$,
  its image lies in $M$'s fixed points; since they are by assumption
  isolated, the map is constant.  But then $\beta = 0$, contrary
  to assumption.

  Now we know
  $\wt T$ acts nontrivially on each component $C\iso \PP^1$ of $\sigma$.
  The maximal tori in $Aut(C) \iso PGL_2$ are conjugate, and each
  fixes exactly two points. So $C$ touches at most two nodes,
  hence can be connected to at most two
  other components. Being connected, the only options for $\Sigma$ is
  a chain of components (where the end components each connect to
  only one other) or a ring (where every component connects to two others).
  Such a ring would be genus one, and is thus excluded. This establishes
  the first claim.

  Let $C$ be one component of $\Sigma$; the local co\"ordinates
  around the $\wt T$-fixed points ``$0$'' and ``$\infty$'' are related
  by $z\mapsto z^{-1}$, hence the isotropy weights are negative one
  another, say $\alpha$ and $-\alpha$. If another component $D$ is
  glued to $C$ at $0$, then its tangent weight there will be $-\alpha$
  by proposition \ref{prop:dual}. Hence it too will have isotropy
  weights $\pm \alpha$, so by induction and connectedness all
  components of $\Sigma$ will have isotropy weights $\pm \alpha$ and
  hence will be $\wt T$-isomorphic.

  For the no-collapsing statement, recall that in a {\em stable} map
  $\gamma\colon \Sigma \to M$ a component $C \subseteq M$ can only
  collapse if $C$ meets at least three other components. Caterpillars
  have no such components $C$.
\end{proof}

Given an isomorphism $T \iso \Gm$, whose weights are therefore
indexable by $\ZZ$, we can distinguish the two
$T$-fixed points $p$ on a caterpillar that are {\em not} nodes by whether
the $\Gm$ is acting with positive or negative weight on $T_p \Sigma$,
and call them the \defn{back} or \defn{front} of the caterpillar
respectively. We record a couple of results about them:

\begin{Lemma}\label{lem:ends}
  Continue the setup of corollary \ref{cor:chain}, but with a fixed
  isomorphism $T \iso \Gm$. 
  Let $[\Sigma,\gamma] \in M//_{StMap}\, T$
  with $b,f$ the back and front of the caterpillar $\Sigma$.

  \begin{enumerate}
  \item Let $\overline \beta \in A^{\wt T}_1(M)$ be the {\em equivariant}
    class $\gamma_*([\Sigma])$. Then the point restrictions
    $\overline\beta|_p$, $p\in M^T$ are nonzero exactly for
    $p \in \{\gamma(b),\gamma(f)\}$.
  \item 
    Also, if the two Bia\l ynicki-Birula decompositions defined from $T$
    are stratifications, then $\gamma(\Sigma)$ lies in 
    $$ \overline{ \left\{m \in M\colon
        \lim_{z\cdot 0} z\cdot m = \gamma(b) \right\} } \ \cap\
    \overline{ \left\{m \in M\colon
        \lim_{z\cdot \infty} z\cdot m = \gamma(f)  \right\} } $$
  \end{enumerate}
\end{Lemma}

\begin{proof}
  Since our goal is to compute elements in the domain $A_T(pt)$
  (a polynomial ring), we can safely localize,
  i.e. tensor all rings considered with the fraction field of $A_T(pt)$.
  Let $\alpha$ be the $\wt T$-weight on $T_b \Sigma$;
  then (using corollary \ref{cor:chain})
  on any component $C$ of $\Sigma$ with back and front $b_C,f_C$,
  we have $[C] = ([b_C] - [f_C])/\alpha$. Summing over the components,
  all contributions cancel except from the ends, giving
  $[\Sigma] = ([b]-[f])/\alpha$. Then
  $$ \gamma_*([\Sigma])|_p 
  = \left(\prod \{\wt T\text{-weights in }T_p M\} \right) \cdot
  \begin{cases}
    +1/\alpha & p = f \\
    -1/\alpha & p = b \\
    0 & \text{otherwise}
  \end{cases} $$
  and by the assumption $M^T$ discrete, the product upfront
  of isotropy weights is $\neq 0$.

  For the second claim, we will show by induction that
  $\gamma(C) \subseteq \overline{ \left\{m \in M\colon
      \lim_{z\cdot 0} z\cdot m = \gamma(b) \right\} }$
  for each component $C$ of $\Sigma$; the containment in
  $\overline{ \left\{m \in M\colon
      \lim_{z\cdot \infty} z\cdot m = \gamma(f) \right\} }$ follows from the same
  argument but reversing the parametrization of $T$.

  Let $M_p^\circ := \{m \in M\colon \lim_{z\to 0} z\cdot m = p\}$
  denote the B-B stratum. By the stratification assumption,
  if $\lim_{z\to 0} z\cdot m \in \overline{M_p^\circ}$
  then $m\in \overline{M_p^\circ}$. Hence when $C$ is a component $\Sigma$
  with back and front $b_C,f_C$, 
  for us to show $\gamma(C) \subseteq \overline{M_p^\circ}$,
  it suffices to show that $\gamma(b_C) \in \overline{M_p^\circ}$.
  
  The base case of the induction is that $C$ is the back component of $\Sigma$,
  i.e. the component containing $b$. Then $b_C = b$, so
  $\gamma(b_C) \in \overline{M_{\gamma(b)}^\circ}$, hence as just argued
  $\gamma(C) \subseteq \overline{M_{\gamma(b)}^\circ}$.

  From there we learn $\gamma(f_C) \in \overline{M_{\gamma(b)}^\circ}$.
  But the front $f_C$ is the back $b_{C'}$ of the next component,
  which is the induction step allowing us to eventually learn that
  every component lands inside $\overline{M_{\gamma(b)}^\circ}$.
\end{proof}

It seems likely that the stratification hypothesis in the lemma
is unnecessary, but the argument looked harder as well as unnecessary for our
main examples.

\junk{
A variant of the ``equivariantly smoothable'' condition lets us forego these
extensions $\wt T$ of $T$, which doesn't affect the proofs much but simplifies
the later notation a bit.

\begin{Proposition}\label{prop:janda}
  Let $T\actson M$, and $[\Sigma,\gamma] \in StMap_\beta(\PP^1,M)^T$
  lie in a geometric component $C$ containing curves $[\Sigma',\gamma']$
  such that
  \begin{enumerate}
  \item $\gamma'$ is injective
  \item the $T$-action on $M$ lifts to $\Sigma'$ (i.e. no extension
    $\wt T$ is necessary), with $\gamma'$ $T$-equivariant.
  \end{enumerate}
  Then the $T$-action lifts likewise to $\Sigma$, with $\gamma$ $T$-equivariant.
\end{Proposition}

\begin{proof}
  Let $\wt T$ be the extension of $T$ that acts on the component $C$.
  By assumption, the action of $\wt T$ on the
  coarse moduli space is trivial, the subject of \cite[Theorem 1.1]{AlperJanda}.
  In part 1 of that theorem they find there is a fixed finite group
  (the kernel of $\wt T \onto T$) acting as automorphism group of {\em every} point
  in $C$. By checking at $\gamma'$, one sees that this finite group is trivial.
\end{proof}
} % end junk

\subsection{The choice of $\beta$}\label{ssec:beta}

In the rest of this section we fix
a ``circle'' action $S \colon \Gm \to Aut(M)$ on a projective variety.
This action defines two opposed Bia\l ynicki-Birula decompositions,
each of which has one dense stratum (by $M$'s irreducibility);
let $U \subseteq M$ be the intersection of the two open sets.
% \junk{
%   Fix also a pair $(\lambda,\mu)$ of isolated $S$-fixed points on $M$,
%   and assume that the intersection
%   $$ U := \{m \in M\colon \lim_{z\to 0} z\cdot m = \lambda,\
%   \lim_{z\to \infty} z\cdot m = \mu \} $$
%   of opposed Bia\l ynicki-Birula strata is irreducible\footnote{The
%     empty set is neither irreducible nor reducible, much as $1$ is
%     neither prime nor composite. Also, there are non-smooth toric examples
%     in which the intersection of opposed B-B strata is reducible.}
%   (as is true in our main application, $M = G/P$).
% } % end junk
Then by properness of $M$ the map $\Gm \to M$, $z\mapsto z\cdot m$
extends uniquely to a map $\gamma\colon \PP^1 \to M$,
and for general points $m \in U$, the Chow class
$\beta := \gamma_*([\PP^1]) \in A_1(M)$ is independent of $m$.

% Our next step in the construction involves choosing a $\beta$.
In our principal application, $M$ is a flag manifold $G/P$,
and we take as our circle $\rhocek\colon \Gm \to T$
the smallest regular dominant coweight.
That is, $\rhocek$ acts on each simple root space in $\lie g$ with weight $1$.
When we need a cover of it (as in proposition \ref{prop:extension}),
we will call it $\rhotilde$.

There is a slightly different choice $\beta := \left[\,
  \overline{\Gm\cdot m}\,\right]$, which agrees with the previous if and when
$m$ is selected with trivial $S$-stabilizer. Our choice
$\beta := \gamma_*([\PP^1])$ will be slightly more convenient
(and usually they are equal).
Also, exactly the same $\beta$ is chosen in the definition of
``Chow quotient'', which we recall in \S\ref{sec:Chow}.

{\em Example: $G=SL_3$, $G/P = \PP^2$, $\lambda = [*,0,0]$, $\mu = [0,0,*]$.}
In this case the action is $\rhocek(z) \cdot [a,b,c] = [z^{-2}a, b, z^2c]$.
The open Richardson variety $U$ is $\{[a,b,c]\colon a,c\neq 0\}$.
Not every point in it has the same stabilizer:
the $\rhocek$-orbit closure through
$[1,0,1]$ is the degree $1$ curve $\{[*,0,*]\}$, whereas a general
$\rhocek$-orbit closure $\overline{\rhocek\cdot [a,b,c]}$ (no co\"ordinate
being zero) is the degree $2$ curve $\{[p,q,r]\colon q^2\, ac = pr\, b^2\}$.
So in this example we want $\beta = 2[\PP^1]$.

\subsection{The stable map quotient
  $M //_{StMap}\, S$}
\label{ssec:simpler}

Let $U \subseteq M$ be the open set from the last subsection.
Then we get an assignment
\begin{eqnarray*}
  \varphi_M\colon\qquad U&\to& Map_\beta(\PP^1, M)^S \qquad\qquad\qquad
                               \qquad\qquad\qquad 
          \into StMap_\beta(\PP^1, M)^S \\
  u &\mapsto& \left[\Sigma = \PP^1, \quad \gamma\colon
                 \begin{array}{ccr}
                   \PP^1 &\to& m \\
                   z\in\Gm &\mapsto&  S(z)\cdot u
                 \end{array}
                                     \, \right]
\end{eqnarray*}
(although $\gamma$'s definition is only written on $\Gm$,
the map $\gamma$ extends
uniquely to $\PP^1$, because $M$ is proper).
%projective hence closed and separated).
The image $\phi_M(U)$ is irreducible,
hence, lies in a unique connected component of
$StMap_\beta(\PP^1, M)^S$. Call that component the \defn{stable map quotient},
and take it for our space $M //_{StMap}\, S$.
There is one caveat: if $M$ is a point, this definition read word-by-word
will produce a point as well, but instead we prefer to define $pt //_{StMap}\, S$
to be empty. With that in place, $\dim (M //_{StMap}\, S) = \dim M - 1$.

In \S\ref{sec:Chow} we will see that the coarse moduli space of
$M //_{StMap}\, S$ is the {\em Chow quotient} $M //_{Chow}\, S$, which
is a type of quotient that exists also for higher-dimensional group actions.
We, however, only define the stable map quotient for $\Gm$-actions
(though perhaps one could, as in \cite{Alexeev},
extend the definition using ideas from the Minimal Model Program).

\begin{Theorem}\label{thm:smooth}
  Let $X_\lambda^\mu \subseteq G/P$ be a Richardson variety, typically singular.
  Nonetheless, $X_\lambda^\mu //_{StMap}\, \rhocek$ is a smooth orbifold.
\end{Theorem}

\begin{proof}
  Let $\iota\colon X_\lambda^\mu \into G/P$ be the injection,
  which induces an injection 
  $StMap_{\beta}(\PP^1, X_\lambda^\mu) \into StMap_{\iota_*(\beta)}(\PP^1, G/P)$
  and from there an injection on the $\rhocek$-fixed point sets.
  
  Since $StMap_{\iota_*(\beta)}(\PP^1, G/P)$ is a smooth orbifold, so too is
  $StMap_{\iota_*(\beta)}(\PP^1, G/P)^\rhocek$, the $\rhocek$-fixed point set.
%  (although its components may have different dimensions).
  Let $M$ denote the connected component (which, by smoothness, is
  also a geometric component) containing the image of
  $X_\lambda^\mu //_{StMap}\, \rhocek$.
  It will then suffice to prove that $M$ is, in fact, that image.

  Consider the function
  $StMap_{\iota_*(\beta)}(\PP^1, G/P)^\rhocek \to A_1^\rhocek(G/P)$
  to the {\em equivariant} Chow group, taking
  $[\Sigma,\gamma] \mapsto \gamma_*([\Sigma])$.
  It is locally constant,
  and we let $\overline\beta$ be its (constant) value on $M$,
  which we can determine by evaluating on elements of $\phi_M(u)$
  (as defined at the beginning of the section).
  Then use lemma \ref{lem:ends}(1) to determine $\gamma(b),\gamma(f)$
  for the back $b$ and front $f$ of {\em any} $[\Sigma,\gamma] \in M$,
  and lemma \ref{lem:ends}(2) to show that
  $\gamma(\Sigma) \subseteq X_\lambda \cap X^\mu$.
\end{proof}

\junk{
  It is easy to see that
  $M_\lambda^\mu//_{StMap}\, S = {\overline U}_\lambda^\mu//_{StMap}\, S$,
  and as such, one might reasonably wonder why we involve $M$
  in the notation. We did so because we didn't see any reason
  not involving $G/P$ itself for the (usually) singular Richardson variety
  $X_\lambda^\mu$ to have a smooth stack as stable map quotient by $\rhocek$.
} % end junk

This is not yet the space we'll use to resolve $X_\lambda^\mu$; in
particular, $X_\lambda^\mu//_{StMap} \,\rhocek$ has dimension
one smaller than $X_\lambda^\mu$ has. We ``fix'' this now,
using a standard trick.

\subsection{Marking a component using the graph space}
\label{ssec:marking}

Extend the circle action $S\actson M$ to one on $M\times \PP^1$
using the standard (weight $1$) action $\Gm\actson \PP^1$.
Lift $\lambda,\mu \in M^S$ to $(\lambda,0)$, $(\mu,\infty)$.
We briefly need to distinguish the class $\beta$
of a generic $S$-orbit closure on $M$ from the class $\beta'$
of a generic $S$-orbit closure on $M\times \PP^1$.
Identifying $A_1(M\times \PP^1) \iso A_1(M)\times \ZZ$,
it is easy to show that $\beta' \mapsto (\beta,1)$,
at which point we can abandon the notation $\beta'$.

We are now considering the space $StMap_{(\beta,1)}(\PP^1, M \times \PP^1)$,
standardly called a ``graph space'' as each element of it
contains (and is usually equal to) the graph of some map $\PP^1 \to M$.
This is because for any $[\Sigma,\gamma]$ in the graph space, the composite
$$ \Sigma \to M \times \PP^1 \onto \PP^1
$$
is degree $1$ by assumption; hence, exactly one component
$C$ of $\Sigma$ maps isomorphically to the target, while all other
components of $\Sigma$ collapse to points.
% Call this $C$ the \defn{marking component}.

Pull back a point $w \in \PP^1$ along this isomorphism $C \to \PP^1$, 
and then project instead to $M$.
Chaining these constructions together we get a map
$$
ev_w \colon StMap_{(\beta,1)}(\PP^1, M \times \PP^1) \to M.
$$
We called it ``$ev$'' as it is reminiscent of the evaluation maps
out of stable mapping spaces for curves with marked points.
We do {\em not} use marked points here but instead, effectively,
mark a whole component. (This is not quite fair, as marking $n$
points on $\Sigma$ involves a function $[n]\into \Sigma$, rather than our
function $\Sigma \onto \PP^1$.)

\begin{Theorem}\label{thm:ev}
  For $w\in \Gm$, the map $ev_w$ is algebraic. For $w=0,\infty$
  it is constructible.
\end{Theorem}

\begin{proof}
  Let $F$ be the universal family over $StMap_{(\beta,1)}(\PP^1,M \times \PP^1)$,
  whose fibers are the curves $\Sigma$, and let $\Gamma$ be the universal
  map $F \to M \times \PP^1$. By generic smoothness in characteristic $0$
  (the algebraic Sard's theorem), for some point $w \in \Gm$
  the fiber of $F \to M\times \PP^1 \onto \PP^1$ is smooth.
  Then by $\Gm$-equivariance, this smoothness holds for every $w\in\Gm$.
    
  For the constructibility, we stratify $StMap_{(\beta,1)}(\PP^1,M \times \PP^1)$
  according to the dual graph of the curve $\Sigma$. Over each stratum,
  the universal family breaks into components, exactly one of which
  projects isomorphically to $\PP^1$, trivializing that subfamily.
\end{proof}

\subsection{Combining these, to resolve $X_\lambda^\mu$}
\label{ssec:combining}

Let $X_{(\lambda,0)}^{(\mu,\infty)} \subseteq (G \times SL_2)/(P \times B_{SL_2})$,
and define %our space
$$ \wt{X_\lambda^\mu} 
:= X_{(\lambda,0)}^{(\mu,\infty)}//_{StMap}\, \rhocek_\Delta $$
where $\rhocek_\Delta$ is the smallest regular dominant coweight
of $G \times SL_2$.
\junk{
  i.e. the closure of the image of the first map
  $$
  \begin{matrix}
    U \text{ (from \S\ref{ssec:simpler})}
    &\xrightarrow {\varphi_{G/P\times \PP^1}}
    & StMap_{(\beta,1)}(\PP^1, G/P \times \PP^1)^\rhocek
    &\xrightarrow {ev_1} & G/P \\
    % gP/P &\mapsto& \overline{\forall z \in \Gm,\
    % z \mapsto (\rhocek(z)\cdot gP/P,z)}
    % &\mapsto& gP/P.
  \end{matrix}
  $$

We recapitulate our prior conclusions. Since $U$ is irreducible,
$\varphi_{G/P\times \PP^1}$ has picked out one
irreducible component of $StMap_{(\beta,1)}(G/P\times \PP^1)^\rhocek$.
Since $StMap_{(\beta,1)}(G/P\times \PP^1)$
is a smooth orbifold, so too is the fixed-point set
$StMap_{(\beta,1)}(G/P\times \PP^1)^\rhocek$
and each connected component thereof, in particular $\wt{X_\lambda^\mu}$. 
Since the composite map is the inclusion $U \into G/P$, we know
$ev_1\colon \wt{X_\lambda^\mu} \to G/P$ is birational to its image.

The closure of the image of $U$ in $G/P$ is of course $X_\lambda^\mu$.
}% end junk

We recapitulate our prior conclusions. By theorem \ref{thm:smooth}
this space is a smooth orbifold. By theorem \ref{thm:ev} it has an
algebraic map $ev_1$ to $X_\lambda^\mu$. On the dense set where $\Sigma$
is irreducible, its image in $X_\lambda^\mu \times \PP^1$ is the graph
of a $\rhocek$-equivariant map $\PP^1 \to X_\lambda^\mu$, which is
determined by the image of $1\in \PP^1$; hence the map is birational.

In all, we have an orbifold resolution of singularities of $X_\lambda^\mu$,
with no choices involved -- other than $S = \rhocek$, but that is
a canonical choice. (It does mean though that the resolution
$\wt{X_\lambda^\mu} \onto X_\lambda^\mu$ is only equivariant for the
group $Z_G(\rhocek) = T$, rather than for $Aut(X_\lambda^\mu)$.)

\junk{
  Since $0,\infty$ are $\rhocek$-fixed on $\PP^1$, their evaluation maps
  $ev_0,ev_\infty\colon \wt{X^\mu_\lambda} \to G/P$ land inside the
  discrete set $(G/P)^\rhocek \iso W/W_P$, and hence are constant.
  Indeed, $ev_0 \equiv \lambda$ and $ev_\infty \equiv \mu$.
}

\junk{
  There is a very funny aspect of this construction -- while we start
  with stable maps into $G/P$, eventually we're only looking at
  stable maps into $X_\lambda^\mu$. In particular, the definition of the
  space could have begun with $StMap_\beta(\PP^1,X_\lambda^\mu)$ and
  culminated in the same space. However, our proof that $\wt{X_\lambda^\mu}$
  is {\em smooth} springs from the smoothness (and convexity) of $G/P$,
  and we did not see as easy a way of proving the smoothness directly.
}

\subsection{Maps between stable map quotients}
\label{ssec:maps}

Given $f\colon X\to Y$ proper and a curve class $\beta \in A_1(X)$,
we get a map of stacks $StMap_\beta(\PP^1,X) \to StMap_{f_*\beta}(\PP^1,Y)$
taking $[\Sigma,\gamma]$ to $[\Sigma', f \circ \gamma]$ where
$\Sigma'$ is the ``stabilization'' of $\Sigma$.
Specifically, any component $C$ of $\Sigma$ meeting fewer than three
other components, and contracted by $f\circ\gamma$ (though not by
$\gamma$, since $[\Sigma,\gamma]$ was stable) is collapsed to a point
in $\Sigma'$.

Let $f\colon X\onto Y$ be $S$-equivariant, and define curve classes
$\beta_X,\beta_Y$ on them as in \S\ref{ssec:beta}. If we take a general point
$y\in Y$, and a general point $x \in f^{-1}(y)$, we can show easily that
$f_*([\beta_X]) = \beta_Y$. Now define the function
$$ f //_{StMap}\, S\colon \ \ X //_{StMap}\, S \to Y //_{StMap}\, S $$
taking $[\Sigma,\gamma]$ to $[\Sigma',f \circ \gamma]$ where
$\Sigma' = \Sigma/\!\!\sim$ is the smaller caterpillar formed by
collapsing those components of $\Sigma$ collapsed by $f\circ \gamma$.
The only difference from the general (non-$S$-equivariant)
case is that we know that every component meets fewer than three
other components and must be contracted.

\subsection{Stratifying by number of nodes}
\label{ssec:stratNodes}

\junk{
  We investigate the simplest local picture, $V \times \PP V \onto V$
  where $V = \CC_1 \oplus \CC_{-1}$ as a representation of $S$.
  While the space of stable maps from a noncompact source curve is
  ill-behaved, our $S$-equivariant maps are easy to consider.
  The orbit through a general point looks like
  $$ S \cdot (a,b,[c,d]) = \{ (sa,s^{-1}b, [sc,s^{-1}d]) \}
  = \{ (p,q, [r,s]) \colon pq = ab, [qr, pb] = [bc,ad] \} $$
  so this moduli space $Map(\PP^1, V\times \PP V)^S$ compactifies to
  $\CC \times \PP^1$. Meanwhile,
}

We record a first description of the stratification of $M//_{StMap}\, S$
by number of nodes, with further analysis to come in \S\ref{sec:dualcomplex}.

\begin{Proposition}\label{prop:sections}
  Fix a group isomorphism $S \iso \CC^\times$, and
  let $S$ act on a projective variety $M$ with isolated fixed points.
  \junk{Let $\wt S$ be the finite extension
    of $S$ that acts on each $[\Sigma,\gamma] \in M_\lambda^\mu//_{StMap}$.}
  Let $(M_\lambda^\mu//_{StMap}\,S)_{k\text{ nodes}}$ be the locally closed
  substack consisting
  of curves with exactly $k$ nodes. Then each curve has exactly $k+2$ many
  $\wt S$-fixed points, which can be consistently numbered $0\ldots k+1$,
  and the ``$i$th fixed point'' defines an algebraic section of the
  universal family restricted to $(M_\lambda^\mu//_{StMap}\,S)_{k\text{ nodes}}$.
  
  In particular, there is a map $\nu_i \colon
  (M_\lambda^\mu//_{StMap}\,S)_{k\text{ nodes}} \to M$
  taking $[\Sigma,\gamma] \mapsto \gamma($the $i$th fixed point$)$,
  and its image lies in $M^S$. The maps $\nu_0,\nu_{k+2}$ are
  constant maps, whose values are the $S$-sink and $S$-source
  (in the sense of Bia\l ynicki-Birula decompositions) respectively.
\end{Proposition}

\junk{
  Thanks to proposition \ref{prop:janda},
  each $\gamma\colon \Sigma \to M$ is already $S$-equivariant;
  we don't need to extend $S$ to get an action.
}

\begin{proof}
  Recall from corollary \ref{cor:chain} that $\Sigma$ is a caterpillar;
  since it has $k$ nodes, it has $k+1$ components
  and two other $S$-fixed points, the front and back. 
  Number the back $0$, and the other $S$-fixed point on its component $1$.
  If $1$ is a node, then it connects to another component;
  number the other $S$-fixed point on that component $2$, and so on.

  Over $(M_\lambda^\mu//_{StMap}\,S)_{k\text{ nodes}}$ the universal family
  has $k+1$ components, numbered by their distance from the back,
  and the intersection of the $(i-1)$st and $i$th component defines
  the point subfamily (i.e. section) $\nu_i$ for $i\neq 0,k+2$.
  The last two arise as the two remaining components of the $\rhocek$-fixed
  points on the universal family.

  Since any node $n \in \Sigma$ must be $S$-fixed,
  its image $\gamma(n)$ lies in $M^S$.
\end{proof}

Define the \defn{boundary}
$\partial(M//_{StMap}\,S) \subset M//_{StMap}\,S $
by the condition ``$\Sigma$ has a node''.
% By corollary \ref{cor:chain}, if $\Sigma$ has $k$ nodes,
% it is a chain of $k+1$ $\PP^1$s, with $k+2$ $S$-fixed points.
This is a standard divisor to consider in such moduli spaces;
on $\overline{\calM_{0,4}} \iso \PP^1$ it consists of $3$ points hence
is not anticanonical. We study $\partial(M//_{StMap}\,S)$'s
anticanonicality in \S\ref{sec:anticanonical}.

\section{Relation to the Chow quotient}\label{sec:Chow}

The construction in \S\ref{ssec:simpler} is extremely similar to that of
the {\em Chow quotient} of $M$ by a group action of $G$ (see e.g.
\cite{Kapranov}).
In that story, one likewise defines a cycle class $\beta$ of general
$G$-orbit closures, and considers $G$-fixed points, but now inside a
Chow variety $\overline{C_{M,\beta}}$ of cycles rather than in a stack
of stable maps.

We have a ``cycle map''
$$ StMap_\beta(\PP^1,M) \to \overline{C_{M,\beta}} $$
taking $[\Sigma,\gamma]$ to its image $\gamma(\Sigma)$ counted with
multiplicities coming from the degrees of the map.
For general curve classes $\beta$ this cycle map is far from one-to-one;
for example when $M = \PP^1$ and $\beta = 2[\PP^1]$, the stable map space
is $(\PP^1)^2/S_2$ (recording the ramification points) whereas the
Chow variety is a point (the entire $\PP^1$ target is painted with a ``$2$'').

However, when we restrict to the locus of equivariantly smoothable
stable curves, the map becomes bijective, as we show now.
Recall from corollary \ref{cor:chain} that for $[\Sigma,\gamma] \in M//_{StMap}\, S$,
each component $C \subseteq \Sigma$ maps $\rhotilde$-equivariantly and
finitely (doesn't collapse) to its image. Hence the ramification
points of $\gamma|_C$ are known (they are the two $\rhocek$-fixed points
on the image), and the ramification degree is known (as the cycle
records that), which is enough to determine the map up to isomorphism. QED.

In fact even the associated cycle is more information than we need:

\begin{Proposition}\label{prop:curvestab}
  Let $[\Sigma,\gamma] \in M //_{StMap}\, S$. Then $[\Sigma,\gamma]$
  is determined by its image $\gamma(\Sigma)$ alone; the multiplicities
  can be inferred. More specifically, the multiplicity of a component
  $C \subseteq \gamma(\Sigma)$ is the order of the generic $S$-stabilizer
  on $C$ divided by that on $M$. 
\end{Proposition}

\begin{proof}
  First observe that if $X,Y$ are two $G$-varieties with dense $G$-orbits,
  then up to isomorphism there is at most one $G$-equivariant map
  $\gamma\colon X\onto Y$. Also, if $X$ is a smooth curve, $Y$ proper,
  and we have containment (up to conjugacy) of the generic $G$-stabilizer
  on $X$ in the generic $G$-stabilizer on $Y$, then the map will exist.
  Proof: fix $x\in X$ in the dense orbit. If
  $\gamma(x)$ does not lie in $Y$'s dense $G$-orbit $U$, then it lies in
  the closed complement $Y\setminus U$, in which case so do
  $\gamma(G\cdot x)$ and $\gamma\left(\overline{G\cdot x} = X\right)$
  by equivariance and density. This violates $\gamma$ being onto.
  The map $x\mapsto y$ extends to $G\cdot x \to G\cdot y$ only if
  $Stab_G(x) \leq Stab_G(y)$, in which case it extends uniquely,
  and then extends in at most one way to the rest of $X$.
  (If $X$ is a smooth curve and $Y$ is proper, then the extension exists.)
  The only choice involved was of $y$ in $Y$'s dense $G$-orbit,
  which conjugates $Stab_G(y)$,
  and different choices give isomorphic objects in the category of
  $G$-spaces over $Y$. QED.
  
  Let $D \iso \PP^1$ be the normalization of a generic $S$-orbit
  closure in $M$, considered as an $S$-variety.
  Since the $G$-stabilizers on $M$ vary upper semicontinuously, the
  previous paragraph says that $D$ has a unique (up to isomorphism)
  map onto each component $C$ of $\gamma(\Sigma)$. On the open orbits,
  the map looks like $\Gm/(\ZZ/m) \onto \Gm/(\ZZ/mn)$, whose degree
  is plainly the ratio of the stabilizer orders.
  \junk{Let $k+1$ be the number of components of $\gamma(\Sigma)$, and glue
    together that many copies of $D$ into a chain $\Lambda$.
    Then we have shown that $\Sigma \iso \Lambda$ as $S$-spaces over
    $\gamma(\Sigma)$.  }
\end{proof}

Perhaps the main reason to work with the space of stable maps rather than
the Chow variety is that the former naturally includes the stack structure,
which in the $X_\lambda^\mu$ case makes the space a ``smooth stack'',
unlike its coarse moduli space the Chow quotient (which is therefore
rationally smooth, hence satisfies rational Poincar\'e duality).
In particular, the tangent spaces of this smooth stack all have
the same dimension, and we study these tangent spaces in the next section.

Although the coarse moduli space of $X_\lambda^\mu//_{StMap}\, S$ is normal
and proper, and maps bijectively to $X_\lambda^\mu//_{Chow}\, S$,
it does not follow that the cycle map is an isomorphism; potentially
the target space has some cuspidal singularities or some nonreduced structure.
I don't know if these issues actually occur. Michalek and Wang have
informed me that in their study of the Chow quotient, they pass to its
normalization to avoid such niceties.

There is a similar bijectivity result in \cite{ChenSatriano}, though
specific to toric varieties and involving log structures.

\section{The stratification on
  $X_\lambda^\mu //_{StMap}\, \rhocek$,
  and its dual simplicial complex}
\label{sec:dualcomplex}

\subsection{Paths in $G/P$}\label{ssec:pathsinGmodP}

Recall from \S\ref{ssec:stratNodes} the locally constant maps
$\nu_i\colon (M_\lambda^\mu//_{StMap}\,S)_{k\text{ nodes}} \to M^S$.
% Using the values of $(\nu_i)$ we can partition
% $(M_\lambda^\mu//_{StMap}\,S)_{k\text{ nodes}}$.
For $\kappa = (\kappa_1,\ldots,\kappa_m)$ a sequence in $M^S$, define
$$
(M//_{StMap}\, S)^\circ_{\kappa} := \left\{ [\Sigma,\gamma] \in M//_{StMap}\, S
  \colon \text{ $\Sigma$ has $m$ nodes, and $\nu_i(\Sigma) = \kappa_i$
    for $i=1,\ldots,m$} \right\}
$$
and let $(M//_{StMap}\, S)_{\kappa}$ denote its closure.
Most sequences give an empty set; see the next theorem.

%\subsection{The $\wt{X_\lambda^\mu}_-$ case}\label{ssec:minuscase}

Take $\lambda \neq \mu$; our Richardson variety is therefore not a point,
and $\beta$ is not $0$ (the better to apply corollary \ref{cor:chain}).

\begin{Theorem}\label{thm:strata}
  Let $\kappa = (\kappa_1,\ldots,\kappa_m)$ be a sequence of length $m$
  in $W/W_P$, and pad its ends with $\kappa_0 := \lambda$, $\kappa_{m+1} := \mu$.

  As we have only analyzed the tangent spaces to
  $X_\lambda^\mu//_{StMap}\, \rhocek$ in the cases of $G/P$
  a Grassmannian or full flag variety, assume $G/P$ is one of those.
  \begin{enumerate}
  \item 
    $(X_\lambda^\mu//_{StMap}\, \rhocek)_{\kappa} \iso
    (X_{\kappa_0}^{\kappa_1}//_{StMap}\, S) \times
    (X_{\kappa_1}^{\kappa_2}//_{StMap}\, S) \times
    \cdots \times
    (X_{\kappa_m}^{\kappa_{m+1}}//_{StMap}\, S)
    $
%    In particular it is nonempty exactly when $\kappa$ is a strictly
%    increasing chain in the open interval $(\lambda,\mu) \subseteq M^S$.
  \item If $\kappa$ is not strictly increasing in $W/W_P$ strong Bruhat order, 
    then $(X_\lambda^\mu//_{StMap}\, \rhocek)_\kappa = \emptyset$.
%    \break    So hereafter assume it is strictly increasing.
  \item 
    $\dim (X_\lambda^\mu//_{StMap}\, \rhocek)_{\kappa} 
    = \dim(X_\lambda^\mu) - (k+1)$.
    That is, its codimension in $X_\lambda^\mu//_{StMap}\, \rhocek$ is $k$.
  \item The divisors $\bigcap_{i=1}^m (X_\lambda^\mu//_{StMap}\, \rhocek)_{(\kappa_i)}$
    transversely intersect, giving $(X_\lambda^\mu//_{StMap}\, \rhocek)_{\kappa}$.
  \item
    $\partial (X_\lambda^\mu//_{StMap}\, \rhocek)$
    is a simple normal crossings divisor.
  \end{enumerate}
\end{Theorem}

Part (4) involves a tangent space calculation we will need before the proof.

\subsection{First analysis of the tangent spaces}
\label{ssec:analyzing}

For any smooth $M$, there are three terms contributing to the tangent space
$T_{[\Sigma,\gamma]} StMap_\beta(\PP^1, M)$
(see e.g. \cite[\S 2]{Kwon}):
\begin{enumerate}
\item[(i)] $H^1\left(\wt\Sigma, T\wt\Sigma[-R]\right)$ where $\wt\Sigma$ is the
  normalization (a disjoint union of $\PP^1$s),
  and $R \subseteq \wt\Sigma$ is the now-pulled-apart nodes.
  This is about moving around, on $\Sigma$,
  the points of intersection of the components.
  This term vanishes if every component of $\Sigma$ has at most $3$ nodes,
  as is of course true for caterpillars.
\item[(ii)]
  $\Oplus_{n\in {\rm nodes}(\Sigma)} T_{n_1}(\wt \Sigma)\tensor T_{n_2}(\wt \Sigma)$
  where $n_1,n_2$ are the points lying over the node. This is about
  deforming the node away. (Note the relation to proposition \ref{prop:dual};
  for the node to be $S$-equivariantly deformable, this $1$-dimensional
  $S$-representation need be trivial.)
\item[(iii)] $\coker\left( H^0(\Sigma;\ T\Sigma)
    \to H^0(\Sigma;\ \gamma^* TM) \right)$.    
  This is about moving $\gamma(\Sigma)$ around in $M$.
  Note, for later calculation, the beginning of the Mayer-Vietoris sequence
  $$ 0 \to H^0(\Sigma;\ \gamma^* TM) 
  \to H^0\!\left(\wt\Sigma;\ \gamma^* TM\right)
  \to H^0(\text{nodes}(\Sigma);\ \gamma^* TM)   % \to H^1(\Sigma;\ \gamma^* TM)
  $$
\end{enumerate}

\begin{proof}[Proof of theorem \ref{thm:strata}]
  \begin{enumerate}
  \item Consider the maps $(X_\lambda^\mu//_{StMap}\, \rhocek)_{\kappa}
    \to X_{\kappa_i}^{\kappa_{i+1}}//_{StMap}\, S$ taking a caterpillar to
    its $i$th component, $i=0\ldots m$. Their product gives an obviously
    injective map to the RHS.

    The map backwards {\em seems} equally simple -- suture the
    caterpillars together -- but that only maps obviously to
    $StMap_\beta(\PP^1,X_\lambda^\mu)^\rhocek$. There is a subtlety
    remaining: is the image in the correct connected component,
    i.e. are the sutured caterpillars equivariantly smoothable?
    We'll need a converse to proposition \ref{prop:dual}.

    This is the question of whether each of the
    $T_{n_1} \wt\Sigma \tensor T_{n_2} \wt\Sigma$ summands in the
    tangent space %(see the beginning of \S\ref{ssec:analyzing})
    defines an
    actual deformation of $[\Sigma,\gamma]$. On $G/P$ we know it does,
    because $G/P$ is convex; we then follow the proof of
    theorem \ref{thm:smooth} to establish the same for $X_\lambda^\mu$.
  \item If $\kappa_{i+1} \not > \kappa_i$,
    then $X_{\kappa_i}^{\kappa_{i+1}} = \emptyset$, contributing a $\emptyset$
    factor to the product. (Note that this involves the case
    $\kappa_i = \kappa_{i+1}$, because we declared we wanted
    $pt//_{StMap}\, \rhocek$ defined as empty.)
  \item Add up the dimensions from (1).
  \item This is clear set-theoretically -- for a curve to pass through
    all the points $\kappa$, it needs to pass through each one.
    To be sure that the intersection is reduced, it is enough to
    check the dimensions of the tangent spaces at the $T$-fixed points.
    Each divisor loses one summand of type (ii),
    and the summands lost are distinct.
  \item The components of $\partial (X_\lambda^\mu//_{StMap}\, \rhocek)$ are
    the divisors $(X_\lambda^\mu//_{StMap}\, \rhocek)_{(\kappa_1)}$,
    $\kappa_1 \in (\lambda,\mu)$.
    The sncd requirement is that any nonempty intersection of these
    should be smooth (in particular reduced) and irreducible.
    The reducedness follows from (4),
    at which point the smoothness and irreducibility follow from (1).
  \qedhere
  \end{enumerate}
\end{proof}

Given an sncd $\partial N$ in a manifold $N$ (or, as here, in an orbifold),
the \defn{dual simplicial complex} \cite{KollarXu,Payne} is defined to have
vertex set $V$ the set of components of $\partial N$,
with $F \subseteq V$ a face iff the intersection of the corresponding
components is nonempty. In general these complexes can look like anything
(\cite{KollarXu} again), but are conjectured to be spheres (or closely related)
when $\partial N$ is anticanonical. The \defn{order complex} of a
poset $P$ has vertices $P$ and $F\subseteq P$ is a face iff $F$ forms
a chain in $P$. By (1) and (3) above, the dual simplicial complex
to $\partial \wt{X_u^v}$ is the order complex of the open interval
$(\lambda,\mu) := \{\nu \in W/W_P\colon \lambda<\nu<\mu\}$
in the Bruhat order on $W/W_P$. 

\begin{Theorem}\label{thm:BjornerWachs}
  In the $G/B$ case, the dual simplicial complex to
  $\partial X_u^v //_{StMap}\, \rhocek$ is homeomorphic to a sphere. In the $G/P$
  case, it is sometimes a sphere and sometimes a ball.
\end{Theorem}

\begin{proof}
  As we have seen, these are statements about the order complex of the
  open interval $(\lambda,\mu)$ in the Bruhat order on $W/W_P$,
  which were established in \cite{BjornerWachs}.
\end{proof}

We discuss the ball vs. sphere question further in \S\ref{sec:anticanonical}.

\section{Isotropy actions and Betti numbers}
\label{sec:Tweights}

In this rather technical section we drill further into the type (iii)
summands from \S\ref{ssec:analyzing}, in the special cases $G/P$ a
Grassmannian or full flag manifold.

Assume that $S\actson M$ is a circle action with isolated fixed points.
We compute the tangent space $T_{[\Sigma,\gamma]} (M//_{StMap}\, S)$
as the $S$-invariant subspace inside the tangent space described
in \S\ref{ssec:analyzing}.
Enumerate $\Sigma$'s components as $C_0 \cup C_1 \cup \ldots \cup C_k$,
as in proposition \ref{prop:sections}, allowing us to write
$H^0\!\left(\wt\Sigma;\ \gamma^* TM\right)
= \Oplus_{i=0\ldots k} H^0\!\left(C_i;\ \gamma^* TM\right)$.
Each node in $\Sigma$ is $S$-fixed,
thus maps to an $S$-fixed point in $M$; by the assumption $M^S$ finite
we know $(T_{\gamma(f)}M)^S = 0$. Hence the left Mayer-Vietoris map
$H^0(\Sigma;\ \gamma^* TM) \to H^0\!\left(\wt\Sigma;\ \gamma^* TM\right)$
becomes an isomorphism on the $S$-invariants. In all,
$$
T_{[\Sigma,\gamma]} M//_{StMap}\, S \quad\iso\quad
\Oplus_{n\in\text{nodes}(\Sigma)} T_{n^-}(\wt \Sigma)\tensor T_{n^+}(\wt \Sigma)
\quad\oplus\quad \Oplus_{i=0\ldots k} H^0\!\left(C_i;\ \gamma^* TM\right)^S
$$
where $n^-,n^+ \mapsto n$ under the normalization $\wt\Sigma \to \Sigma$.
Note that the first summands are already $S$-invariant,
by proposition \ref{prop:dual}.
If the map $\gamma$ is equivariant w.r.t. some larger torus $\wt T \geq S$,
then this isomorphism is $\wt T$-equivariant (or, one might say,
$\wt T/S$-equivariant).

To compute the latter summands, we need the following technical lemma.
One way in which it is technical is that we describe the weights
of a finite extension $\wt T$ of $T$ using rational combinations of
those of $T$; expressions like this are possible due to
the inclusion $T^* \into \wt T^*$ of weight lattices being
rationally an isomorphism.

\begin{Lemma}\label{lem:Sinvariants}
  Let $T\actson M$ smooth projective with isolated fixed points.
  Let $\wt T$ be a finite extension of $T$.
  Let $S \leq \wt T$ be a one-parameter subgroup (i.e. with a fixed
  isomorphism $S \iso \CC^\times$) such that $M^S$
  is again finite. Let $C\iso \PP^1$ carry a $\wt T$-action such that
  $S$ acts with trivial generic stabilizer.
  In particular $C^S = \{0,\infty\}$, ordered so that $S$ acts
  on $T_0 C$ with weight $+1$. Let $\gamma\colon C \to M$ be $\wt T$-equivariant.
  
  Assume $\gamma(C)$ is normal,\footnote{One can easily modify the
    argument to work instead with the normalization, but this isn't
    needed in our motivating examples, and seemed unnecessarily distracting.}
  so $TM|_{\gamma(C)}$ is a sum of line bundles. Assume that $M$ is convex
  i.e. that these line bundles are nonnegative. There is no harm
  in replacing this list
  by a $K_T$-equivalent sum $\Oplus_i \calL_i$ of nonnegative line bundles.
  (Such $K_T$-equivalences are easy to verify when true;
  it suffices that $\Oplus_i \calL_i$ have the same isotropy $T$-weights
  at $\gamma(0),\gamma(\infty)$ as $TM$ has.)

  Let $\calL_i|_0$, $\calL_i|_\infty$ denote the restrictions
  of these line bundles to the points $\gamma(0)$, $\gamma(\infty)$.
  For $L$ a $1$-dimensional $T$-representation, write
  $wt_{T}(L) \in T^*$ for its weight, likewise $wt_S(L) \in S^* \iso \ZZ$.

  Then the $\wt T$-representation $H^0(C;\ \gamma^*TM)^S$ can be written
  as the following direct sum over some $\{\calL_i\}$.
  For each $\calL_i$, if $0 \in [wt_S(\calL_i|_0), wt_S(\calL_i|_\infty)]$,
  then $H^0(C;\ \gamma^*\calL_i)^S$ is $1$-dimensional with $\wt T$-weight
  $$ \frac{1}{dm} \bigg(
  wt_S(\calL_i|_\infty) wt_{T}(\calL_i|_0)
  - wt_S(\calL_i|_0) wt_{T}(\calL_i|_\infty)
  \bigg) $$
  where $d$ is the size of the generic stabilizer of $S$ on $\gamma(C)$
  and $m = \deg \calL_i$.
  If however $0$ does not lie in that interval, then
  $H^0(C;\ \gamma^*\calL_i)^S = 0$.
\end{Lemma}

In the applications the composite $S \into \wt T \onto T$ is injective,
but this condition doesn't appear in the proof so we didn't assume it.

\begin{proof}
  If $M$ is convex, then instead of computing $H^0(\Sigma;\ \gamma^* TM)$
  we can equivalently compute the sheaf Euler characteristic. That
  depends only on the $K_T$-class, which allows us to work with any other
  $K_T$-equivalent sum of line bundles that have no higher sheaf cohomology.
  Since $M$ is smooth projective we know $K_T(M) \into K_T(M^T)$ is
  injective, establishing the parenthetical claim in the lemma's statement.

  As in the proof of proposition \ref{prop:curvestab}, the $S$-equivariance
  forces the degree of $\gamma$ to be the size of the generic stabilizer
  of $S$ on $\gamma(C)$ (divided by that on $C$, which was assumed to be $1$).
  We called that degree ``$d$'' in the statement, and so,
  we have learned $\deg \gamma = d$.

  Let $m = \deg(\calL_i)$, which by assumption is $\geq 0$.
  Then $\deg(\gamma^* \calL_i) = dm \geq 0$, with a $(dm+1)$-dimensional
  $H^0$ (and no $H^1$). The $\wt T$-weights on
  $H^0(C;\, \gamma^*\calL_i)$ are of the form
  $$ f \  wt_{T}(\calL_i|_0) + (1-f) \  wt_{T}(\calL_i|_\infty), \qquad
  f \in [0,1] \cap \frac{1}{dm}\ZZ $$
  For such a $1$-dim representation to be $S$-invariant, we need the
  corresponding convex combination of $S$-weights to come out $0$.
  That is achievable with
  $f \in [0,1]$ if and only if $0 \in [wt_S(\calL_i|_0), wt_S(\calL_i|_\infty)]$.
  In detail,
  $$ f \  wt_{S}(\calL_i|_0) + (1-f) \  wt_{S}(\calL_i|_\infty) = 0
  \quad\text{hence}\quad
  f = \frac{wt_{S}(\calL_i|_\infty)} {wt_{S}(\calL_i|_\infty) - wt_{S}(\calL_i|_0)}
  = \frac{wt_{S}(\calL_i|_\infty)} {dm}
  $$
  giving the promised $\wt T$-weight
  $ \frac{wt_{S}(\calL_i|_\infty)} {dm}  wt_{T}(\calL_i|_0)
  - \frac{wt_{S}(\calL_i|_0)}{dm} \  wt_{T}(\calL_i|_\infty) $.
\end{proof}

(It seems remarkable that the condition $f \in \frac{1}{dm}\ZZ$
holds automatically.) To sum up:

\begin{Theorem}\label{thm:Tspace}
  Let $T \actson M$ smooth projective with $M^T$ finite,
  and $S \leq T$ be a circle subgroup with $M^S$ also finite.
  Let $[\Sigma,\gamma] \in (M //_{StMap}\, S)^T$, and order the
  components $\Sigma = C_0 \cup_{n_1} C_1 \cup_{n_2} \ldots \cup_{n_k} C_k$
  as in proposition \ref{prop:sections}. Let $n_i^-$ resp. $n_i^+$
  be the preimage of $n_i$ in $C_{i-1}$ resp. $C_i$ under the
  normalization map $\wt\Sigma \onto \Sigma$.

  Assume $M$ is convex, that each $\gamma(C_i)$ is normal,
  and that $TM|_{\gamma(C_i)}$ is $K_T$-equivalent to
  a sum $T\gamma(C_i) \oplus \Oplus_j \calL_{i,j}$ of nonnegative line bundles.
  Let $\wt T$ be an extension of $T$ that acts on $\Sigma$, with $\gamma$
  $\wt T$-equivariant.

  Then $T_{[\Sigma,\gamma]} M //_{StMap}\, S$, as a $\wt T$-representation,
  is the sum of the following weight spaces:
  \begin{enumerate}
  \item For each node $n_i$, the tensor product
    $T_{n_i^-} C_{i-1} \tensor T_{n_i^+} C_{i}$ of tangent spaces
    to the components.
  \item For each line bundle $\calL_{i,j}$
    with $0 \in [wt_S(\calL_{i,j}|_0), wt_S(\calL_{i,j}|_\infty)]$,
    the $1$-dim representation with weight
    $\frac{1}{d_i m_i} \bigg(
    wt_S(\calL_i|_\infty) wt_{T}(\calL_i|_0)
    - wt_S(\calL_i|_0) wt_{T}(\calL_i|_\infty)
    \bigg)$
    where $d_i$ is the size of the generic $S$-stabilizer on $\gamma(C_i)$
    and $m_i = \deg(\calL_i)$. 
  \end{enumerate}
\end{Theorem}

\begin{proof}
  This collects the foregoing results, with the only extra feature being that
  the tangent bundle of each curve is separated out as a summand we
  then ignore. This is in order to handle the cokernel computation in (iii)
  of \S \ref{ssec:analyzing}.
\end{proof}

\subsection{The Grassmannian case}\label{ssec:Gr}

Index the $n\choose k$
$T$-fixed points on $Gr(k,\CC^n)$ by English partitions inside the
$k\times (n-k)$ rectangle, where the Bruhat minimum is the empty
partition.  The $T$-fixed points on
$Gr(k,\CC^n)_\lambda^\mu//_{StMap}\, \rhocek$ are then indexed by
chains $(\nu_0=\lambda,$ $\nu_1,\ldots,\nu_k,$ $\nu_{k+1}=\mu)$ of
partitions, where each $\nu_i = r_\beta\cdot \nu_{i-1} > \nu_{i-1}$.
One way to describe these chains is with \defn{standard rim-hook
  skew-tableaux $\tau$ of shape $\mu\setminus\lambda$}, where a
\defn{rim-hook} is an edge-connected skew-partition containing no
$2\times 2$ box, and the rim-hooks in $\tau$ are numbered
$1\ldots k+1$ from Northwest to Southeast. For example, here are the
$14$ SRHSTx of shape $(3,2)\setminus (1)$, with the ``$a+b$'' captions
to be explained later:
$$
{  \begin{tikzpicture}[scale=0.5]
\draw (0,2) -- (0,3) -- (1,3) -- (1,4) -- (2,4) -- (3,4);
\path (0.5,2.5) node {1} ;
\draw (0,2) -- (1,2) -- (1,3) -- (1,4) -- (2,4) -- (3,4);
\path (1.5,3.5) node {2} ;
\draw (0,2) -- (1,2) -- (2,2) -- (2,3) -- (2,4) -- (3,4);
\path (2.5,3.5) node {3} ;
\draw (0,2) -- (1,2) -- (2,2) -- (2,3) -- (3,3) -- (3,4);
\end{tikzpicture} \atop 2+1}
\quad{\begin{tikzpicture}[scale=0.5]
\draw (0,2) -- (0,3) -- (1,3) -- (1,4) -- (2,4) -- (3,4);
\path (0.5,2.5) node {1} ;
\draw (0,2) -- (1,2) -- (1,3) -- (1,4) -- (2,4) -- (3,4);
\path (1.5,3.5) node {2} ;
\draw (0,2) -- (1,2) -- (1,3) -- (2,3) -- (2,4) -- (3,4);
\path (1.5,2.5) node {3} ;
\draw (0,2) -- (1,2) -- (2,2) -- (2,3) -- (2,4) -- (3,4);
\path (2.5,3.5) node {4} ;
\draw (0,2) -- (1,2) -- (2,2) -- (2,3) -- (3,3) -- (3,4);
\end{tikzpicture} \atop 2+0}
\quad{\begin{tikzpicture}[scale=0.5]
\draw (0,2) -- (0,3) -- (1,3) -- (1,4) -- (2,4) -- (3,4);
\path (0.5,2.5) node {1} ;
\draw (0,2) -- (1,2) -- (1,3) -- (1,4) -- (2,4) -- (3,4);
\path (1.5,3.5) node {2} ;
\draw (0,2) -- (1,2) -- (1,3) -- (2,3) -- (2,4) -- (3,4);
\path (2.5,3.5) node {3} ;
\draw (0,2) -- (1,2) -- (1,3) -- (2,3) -- (3,3) -- (3,4);
\path (1.5,2.5) node {4} ;
\draw (0,2) -- (1,2) -- (2,2) -- (2,3) -- (3,3) -- (3,4);
\end{tikzpicture} \atop 2+0}
\quad{\begin{tikzpicture}[scale=0.5]
\draw (0,2) -- (0,3) -- (1,3) -- (1,4) -- (2,4) -- (3,4);
\path (0.5,2.5) node {1} ;
\draw (0,2) -- (1,2) -- (1,3) -- (1,4) -- (2,4) -- (3,4);
\path (2.5,3.5) node {2} ;
\draw (0,2) -- (1,2) -- (2,2) -- (2,3) -- (3,3) -- (3,4);
\end{tikzpicture} \atop 1+1}
\quad{\begin{tikzpicture}[scale=0.5]
\draw (0,2) -- (0,3) -- (1,3) -- (1,4) -- (2,4) -- (3,4);
\path (1.5,3.5) node {1} ;
\draw (0,2) -- (1,2) -- (2,2) -- (2,3) -- (2,4) -- (3,4);
\path (2.5,3.5) node {2} ;
\draw (0,2) -- (1,2) -- (2,2) -- (2,3) -- (3,3) -- (3,4);
\end{tikzpicture} \atop 1+1}
\quad{\begin{tikzpicture}[scale=0.5]
\draw (0,2) -- (0,3) -- (1,3) -- (1,4) -- (2,4) -- (3,4);
\path (1.5,3.5) node {1} ;
\draw (0,2) -- (0,3) -- (1,3) -- (2,3) -- (2,4) -- (3,4);
\path (2.5,3.5) node {2} ;
\draw (0,2) -- (0,3) -- (1,3) -- (2,3) -- (3,3) -- (3,4);
\path (0.5,2.5) node {3} ;
\draw (0,2) -- (1,2) -- (1,3) -- (2,3) -- (3,3) -- (3,4);
\path (1.5,2.5) node {4} ;
\draw (0,2) -- (1,2) -- (2,2) -- (2,3) -- (3,3) -- (3,4);
\end{tikzpicture} \atop 2+0}
\quad{\begin{tikzpicture}[scale=0.5]
\draw (0,2) -- (0,3) -- (1,3) -- (1,4) -- (2,4) -- (3,4);
\path (1.5,3.5) node {1} ;
\draw (0,2) -- (0,3) -- (1,3) -- (2,3) -- (2,4) -- (3,4);
\path (0.5,2.5) node {2} ;
\draw (0,2) -- (1,2) -- (1,3) -- (2,3) -- (2,4) -- (3,4);
\path (1.5,2.5) node {3} ;
\draw (0,2) -- (1,2) -- (2,2) -- (2,3) -- (2,4) -- (3,4);
\path (2.5,3.5) node {4} ;
\draw (0,2) -- (1,2) -- (2,2) -- (2,3) -- (3,3) -- (3,4);
\end{tikzpicture} \atop 2+0}
$$
$$
{\begin{tikzpicture}[scale=0.5]
\draw (0,2) -- (0,3) -- (1,3) -- (1,4) -- (2,4) -- (3,4);
\path (1.5,3.5) node {1} ;
\draw (0,2) -- (0,3) -- (1,3) -- (2,3) -- (2,4) -- (3,4);
\path (0.5,2.5) node {2} ;
\draw (0,2) -- (1,2) -- (1,3) -- (2,3) -- (2,4) -- (3,4);
\path (2.5,3.5) node {3} ;
\draw (0,2) -- (1,2) -- (1,3) -- (2,3) -- (3,3) -- (3,4);
\path (1.5,2.5) node {4} ;
\draw (0,2) -- (1,2) -- (2,2) -- (2,3) -- (3,3) -- (3,4);
\end{tikzpicture} \atop 1+0}
\quad{\begin{tikzpicture}[scale=0.5]
\draw (0,2) -- (0,3) -- (1,3) -- (1,4) -- (2,4) -- (3,4);
\path (0.5,2.5) node {1} ;
\draw (0,2) -- (1,2) -- (1,3) -- (1,4) -- (2,4) -- (3,4);
\path (2.5,3.5) node {2} ;
\draw (0,2) -- (1,2) -- (1,3) -- (2,3) -- (3,3) -- (3,4);
\path (1.5,2.5) node {3} ;
\draw (0,2) -- (1,2) -- (2,2) -- (2,3) -- (3,3) -- (3,4);
\end{tikzpicture} \atop 1+0}
\quad{\begin{tikzpicture}[scale=0.5]
\draw (0,2) -- (0,3) -- (1,3) -- (1,4) -- (2,4) -- (3,4);
\path (1.5,3.5) node {1} ;
\draw (0,2) -- (0,3) -- (1,3) -- (2,3) -- (2,4) -- (3,4);
\path (1.5,2.5) node {2} ;
\draw (0,2) -- (1,2) -- (2,2) -- (2,3) -- (2,4) -- (3,4);
\path (2.5,3.5) node {3} ;
\draw (0,2) -- (1,2) -- (2,2) -- (2,3) -- (3,3) -- (3,4);
\end{tikzpicture} \atop 1+0}
\quad{\begin{tikzpicture}[scale=0.5]
\draw (0,2) -- (0,3) -- (1,3) -- (1,4) -- (2,4) -- (3,4);
\path (1.5,3.5) node {1} ;
\draw (0,2) -- (0,3) -- (1,3) -- (2,3) -- (2,4) -- (3,4);
\path (2.5,3.5) node {2} ;
\draw (0,2) -- (0,3) -- (1,3) -- (2,3) -- (3,3) -- (3,4);
\path (1.5,2.5) node {3} ;
\draw (0,2) -- (1,2) -- (2,2) -- (2,3) -- (3,3) -- (3,4);
\end{tikzpicture} \atop 1+0}
\quad{\begin{tikzpicture}[scale=0.5]
\draw (0,2) -- (0,3) -- (1,3) -- (1,4) -- (2,4) -- (3,4);
\path (2.5,3.5) node {1} ;
\draw (0,2) -- (1,2) -- (2,2) -- (2,3) -- (3,3) -- (3,4);
\end{tikzpicture} \atop 0+1}
\quad{\begin{tikzpicture}[scale=0.5]
\draw (0,2) -- (0,3) -- (1,3) -- (1,4) -- (2,4) -- (3,4);
\path (2.5,3.5) node {1} ;
\draw (0,2) -- (0,3) -- (1,3) -- (2,3) -- (3,3) -- (3,4);
\path (0.5,2.5) node {2} ;
\draw (0,2) -- (1,2) -- (1,3) -- (2,3) -- (3,3) -- (3,4);
\path (1.5,2.5) node {3} ;
\draw (0,2) -- (1,2) -- (2,2) -- (2,3) -- (3,3) -- (3,4);
\end{tikzpicture} \atop 1+0}
\quad{\begin{tikzpicture}[scale=0.5]
\draw (0,2) -- (0,3) -- (1,3) -- (1,4) -- (2,4) -- (3,4);
\path (2.5,3.5) node {1} ;
\draw (0,2) -- (0,3) -- (1,3) -- (2,3) -- (3,3) -- (3,4);
\path (1.5,2.5) node {2} ;
\draw (0,2) -- (1,2) -- (2,2) -- (2,3) -- (3,3) -- (3,4);
\end{tikzpicture} \atop 0+0}
$$

To apply theorem \ref{thm:Tspace} at the $T$-fixed points of
$X_\lambda^\mu//_{StMap}\, \rhocek$, we need to compute (1) the $\wt T$-weights
on the tangent spaces to the $T$-fixed curves, and (2) the restriction
of the tangent bundle (at least up to $K_T$-equivalence).
We compute these in the following lemma, where we found it more convenient
to index $Gr(k,\CC^n)^T$ by subsets $[n]\choose k$.

\begin{Lemma}\label{lem:grH0}
  Let $C \iso \PP^1$ carry the $\rhocek$-action with trivial generic stabilizer.
  Let $\lambda \in {[n] \choose k}$ with $\lambda \not\ni i<j \in \lambda$,
  and let $\gamma\colon C \to Gr(k,n)$ be the $\rhocek$-equivariant map
  whose image is the $T$-fixed curve connecting the points $\lambda$,
  $(i\leftrightarrow j)\cdot \lambda$.
  Lift the $T$-action on the target curve to a $\wt T$-action on $C$,
  with $\wt T$ a finite extension of $T$. 

  Then the $\wt T$-weight on $T_0 C$ (and negative the $\wt T$-weight on $T_\infty C$)
  is $(-y_i+y_j)/(j-i)$.
  Also, $H^0(\Sigma;\ \gamma^*TGr(k,\CC^n))^\rhocek$ has $\wt T$-weights
  $$ \left\{ (-1)^{[a\in \lambda]}
    \left( - \frac{j-a}{j-i} y_i + y_a - \frac{a-i}{j-i} y_j\right) \colon
  a \in (i,j) \right\} $$
  and one $0$ weight. Each weight occurs with multiplicity $1$.
\end{Lemma}

\begin{proof}
  The $T$-weight on $T_0\gamma(C)$ is standardly calculable as $y_i-y_j$.
  On that curve, $\rhocek$ acts with weight $j-i$,
  so (from proposition \ref{prop:curvestab}) $\gamma$ must be of degree $j-i$. 
  On the $(j-i)$-fold cover $C$ of $\gamma(C)$, the $\wt T$-weight must
  therefore be $(y_i-y_j)/(j-i)$. 
  
  Recall for the next computation that the $T$-weights in
  $H^0(\PP(\oplus_i \CC_{\nu_i});\, \calO(1))$ are $\{-\nu_i\}$,
  the minus owing to $\calO(1)$ being the dual of the tautological bundle.

  For $\mu \in Gr(k,n)^T$, the $T$-weights in $T_\mu Gr(k,n)$ are
  $y_a-y_b$ for $a\notin \mu \ni b$; call this set $w_\mu$. Then
  \begin{eqnarray*}
    w_\lambda \cap w_{(i\leftrightarrow j)\cdot \lambda}
 &=& \left\{y_a-y_b \colon a\notin\lambda \ni b,\ a,b \notin \{i,j\} \right\}\\
    w_\lambda \setminus w_{(i\leftrightarrow j)\cdot \lambda}
 &=& \left\{y_a-y_j \colon a \in \lambda^c \setminus i \right\}
     \cup \left\{y_i-y_b \colon b \in \lambda \setminus j \right\}
     \cup \{y_i-y_j\} \\
 &=& \left\{y_a-y_j \colon a \in \lambda^c \setminus i \right\}
     \cup \left\{(y_i+y_j-y_b)-y_j \colon b \in \lambda \setminus j \right\}
     \cup \{y_i-y_j\} \\
    w_{(i\leftrightarrow j)\cdot \lambda} \setminus w_\lambda
 &=& \left\{y_a-y_i \colon a \in \lambda^c \setminus i \right\}
     \cup \left\{y_j-y_b \colon b \in \lambda \setminus j \right\}
     \cup \{y_j-y_i\} 
 \\ &=& \left\{y_a-y_i \colon a \in \lambda^c \setminus i \right\}
     \cup \left\{(y_j+y_i-y_b)-y_i \colon b \in \lambda \setminus j \right\}
     \cup \{y_j-y_i\} 
  \end{eqnarray*}
  Take this first row and three columns
  as defining four sets of $T$-equivariant line bundles:
  \begin{enumerate}
  \item trivial line bundles $\calO(0)$
    with weights $y_a-y_b$, $a\notin \lambda \ni b$, $a,b\notin \{i,j\}$,
  \item copies of $\calO(1)$ on $\PP(\CC_{y_i}\oplus \CC_{y_j})$
    twisted by the character $y_a$, $a \in \lambda^c\setminus i$,
  \item copies of $\calO(1)$ on $\PP(\CC_{y_i}\oplus \CC_{y_j})$
    twisted by the character $y_i+y_j-y_b$, $b \in \lambda\setminus j$,
    \\
    and finally
  \item one copy of the tangent bundle $\calO(2)$ on
    $\PP(\CC_{(\lambda_i-\lambda_j)/2} \oplus \CC_{(\lambda_j-\lambda_i)/2})$.\footnote{%
      It can happen that $(\lambda_i-\lambda_j)/2$ is only a weight
      for a double cover of $T$, i.e. it can happen that $T$ doesn't naturally
      act on $\calO(1)$ of this $\PP^1$. It does, however, act on $\calO(2)$.}
  \end{enumerate}
  At this point we can apply lemma \ref{lem:Sinvariants}.
  Group (1) does not have $0 \in [wt_S(\calL_{i,j}|_0), wt_S(\calL_{i,j}|_\infty)]$.
  In group (2) i.e. $a\notin \lambda$ we're looking for the
  convex combination of $\{-y_i$, $-y_j\}$, plus $y_a$, orthogonal to $\rhocek$.
  In group (3) i.e. $b\in \lambda$ we want the
  convex combination of $\{y_i$, $y_j\}$, minus $y_b$, orthogonal to $\rhocek$.
  So those two agree, up to a sign depending on membership in $\lambda$.
\end{proof}

\tikzset{mynode/.style={circle,draw=black,fill=black,inner sep=1.8pt,outer sep=0pt}}
\tikzset{whitenode/.style={circle,draw=black,fill=white,inner sep=1.8pt,outer sep=0pt}}
\tikzset{edgelabel/.style={\mcol,inner sep=0pt}}
\tikzset{invlabel/.style={draw=black,text=black,circle,inner sep=0pt,minimum size=3mm}}
\newcommand\clopen[3]{
  \draw (#2,-#1) node[mynode] {} -- (#3,-#1) node[whitenode] {} ; }

With these weights in hand we can compute the cell dimensions in a
Bia\l ynicki-Birula decomposition of the coarse moduli space,
obtaining thereby its rational Betti numbers.
This involves a choice of circle $S' \to T$, or equivalently a weighting
of the variables $y_i$; we pick a lexicographic one
$y_1 \gg y_2 \gg \cdots \gg y_n$. That is, a weight
$\sum_k c_k y_k$ is \defn{lex-positive}
based on the sign of the coefficient of the $y_i$ term with $c_i\neq 0$
having smallest $i$.

According to theorem \ref{thm:Tspace}, we have one contribution
$T_{n_i^-} C_{i-1} \tensor T_{n_i^+} C_{i}$ for each {\em intermediate} step
$\lambda_m$ in the chain (not $\lambda$ or $\mu$).
At $(\lambda_{m-1}, \lambda_m, \lambda_{m+1})$, define $a,b,c,d$ by
$\lambda_m = \lambda_{m-1}\cup \{a\}\setminus \{b\}$
and 
$\lambda_{m+1} = \lambda_m\cup \{c\}\setminus \{d\}$.
Then
$$
  wt_{\wt T}(T_{n_m^+}  C_{m}) = \frac{-y_c + y_d}{d-c} \qquad 
  wt_{\wt T}(T_{n_m^-} C_{m-1}) = \frac{ y_a - y_b}{b-a} 
$$
so the total (the weight of the tensor product) is lex-positive iff $a<c$.
(Note that $a=c$ is impossible here, though it will be possible
on flag manifolds, in lemma \ref{lem:flagsH0} to come.)

Again according to theorem \ref{thm:Tspace}, 
the other contributed weights are
$$ \left\{ (-1)^{[a\in \nu]}
  \left( - \frac{j-a}{j-i} y_i + y_a - \frac{a-i}{j-i} y_j\right) \colon
  a \in (i,j) \right\} $$
where $\nu, \nu \cup \{i\} \setminus \{j\}$ are successive steps in
the chain. Such a weight is lex-positive iff $a \notin \nu$.
We record this count:

\begin{Corollary}\label{cor:Grindices}
  Let $(\lambda_i)$ be a SRHST of shape $\mu\setminus \lambda$,
  indexing a $T$-fixed point on $X_\lambda^\mu//_{StMap}\, \rhocek$.
  Then the number of lex-positive weights in the tangent space to that point is
  \begin{enumerate}
  \item the number of rim-hooks numbered $i,i+1$ where the $i$th rim-hook has
    Southwesternmost box in a lower NW/SE diagonal than the $(i+1)$st
    rim-hook does,
    plus
  \item the number of places that a rim-hook could get cut into two
    along a horizontal edge.
  \end{enumerate}
  With these, we can compute $\dim H^{2k}(X_\lambda^\mu//_{StMap}\, \rhocek;\, \QQ)$
  as the number of SRHSTx with exactly $k$ lex-positive weights.
\end{Corollary}

These counts (1)+(2) are the mysterious numbers we placed below the 14
SRHSTx of shape $(3,2)\setminus (1)$ back at the beginning of \S \ref{ssec:Gr}.
Note that their multiplicities are $3^1 2^6 1^6 0^1$, a palindrome; this could
be predicted from the rational smoothness of $X_\lambda^\mu//_{StMap}\, \rhocek$.

One can compute larger examples rather rapidly, e.g. for
$X^{(4,4,2)}_{(2)} //_{StMap}\, \rhocek$ the even Betti numbers are
$ 1, 28, 235, 787, 1167, 787, 235, 28, 1.$

\subsection{The full flag manifold case}

Let $X_u^v \subseteq Fl(n)$ be a Richardson variety inside the
space of full flags. A $T$-fixed point $[\Sigma,\gamma]$
on $X_u^v //_{StMap}\, \rhocek$
is given by an increasing chain
$\gamma(\Sigma^\rhocek) = (u_0 = u < u_1 < u_2 < \ldots < u_m = v)$
in $S_n$ Bruhat order, where each $u_p$ equals $u_{p-1}$ times a
not-necessarily-simple reflection $r_{ij}$, $i<j$.
Call this a \defn{reflection chain}, the full flag analogue of a
standard rim-hook skew tableau.
In this situation we have $\pi(i) < \pi(j)$, and 
$\ell(u_p) - \ell(u_{p-1}) = 1 + 2\#\{a \in (i,j)\colon
\pi(a) \in (\pi(i),\pi(j)) \}$.
That accords nicely with the following calculation and
theorem \ref{thm:strata}(1).

\begin{Lemma}\label{lem:flagsH0}
  Let $C \iso \PP^1$ carry the $\rhocek$-action with trivial generic stabilizer.
  Let $i<j$ in $[n]$ and $\pi\in S_n$ s.t. $\pi(i)<\pi(j)$,
  and $\gamma\colon C\to Fl(n)$ be the $\rhocek$-equivariant map
  whose image is the $T$-fixed curve connecting the points $\pi$,
  $\pi r_{ij}$ where $r_{ij} := (i\leftrightarrow j)$.
  Lift the $T$-action on the target curve to a $\wt T$-action on $C$,
  with $\wt T$ a connected finite extension of $T$. 

  Then the $\wt T$-weight on $T_0 C$ (and negative the $\wt T$-weight on $T_\infty C$)
  is $(-y_i+y_j)/(j-i)$.
  Also, $H^0(\Sigma;\ \gamma^*TFl(n))^\rhocek$ has two $\wt T$-weights
  for each $a \in (i,j)$ such that $\pi(a) \in (\pi(i),\pi(j))$:
  $$
  \pm \left(
    -\frac{\pi(a)-\pi(i)}{\pi(j)-\pi(i)} y_{\pi(j)} + y_{\pi(a)}
    - \frac{\pi(j)-\pi(a)}{\pi(j)-\pi(i)} y_{\pi(i)}
    \right)
  $$
  There is also one $0$ weight. Each weight occurs with multiplicity $1$.
\end{Lemma}

\begin{proof}
  The first calculation is much the same as on $Gr(k,\CC^n)$.

  We again replace the restriction of the tangent bundle
  with a sum of line bundles coming in four groups;
  the first group consists of trivial bundles with no invariant sections,
  two groups consist of degree $1$ line bundles, and the last is one copy
  of the tangent bundle to the $\PP^1$.
  
  For $\pi \in Fl(n)^T$, the $T$-weights in $T_\pi Fl(n)$ are
  $y_{\pi(a)} - y_{\pi(b)}$ for $a<b$; call this set $w_\pi$. Then
  \begin{eqnarray*}
    w_\pi \cap w_{\pi r_{ij}}
    && \text{are all nonzero} \\
    w_\pi \setminus w_{\pi r_{ij}} 
    &=&  \left\{y_{\pi( a)} - y_{\pi( j)}\colon i<a<j \right\}
        \cup \left\{y_{\pi( i)} - y_{\pi( b)}\colon i<b<j \right\}
        \cup \{y_{\pi( i)} - y_{\pi( j)} \} \\ 
    w_{\pi r_{ij}}  \setminus w_\pi
%  &=& (\pi( i)\leftrightarrow \pi(j)) \cdot (w_\pi \setminus w_{\pi r_{ij}}) \\
    &=&  \left\{y_{\pi( a)} - y_{\pi( i)}\colon i<a<j \right\}
        \cup \left\{y_{\pi( j)} - y_{\pi( b)}\colon i<b<j \right\}
        \cup \{y_{\pi( j)} - y_{\pi( i)} \} 
  \end{eqnarray*}
  The line bundle $\calO(1)$ on $\PP(\CC_{y_{\pi(j)}} \oplus \CC_{y_{\pi(i)}})$,
  twisted by $\CC_{\pi(a)}$, has the right isotropy weights to match the
  typical elements in the first column. The $\wt T$-weights of the sections
  of this line bundle, pulled back along $\gamma$, are of the form
  $$ y_{\pi(a)} + f(-y_{\pi(j)}) + (1-f)(-y_{\pi(i)})$$
  for $f\in [0,1].$   For the section to be $\rhocek$-invariant,
  we need ${\pi(a)} + f(-{\pi(j)}) + (1-f)(-{\pi(i)}) = 0$,
  hence, $f = \frac{\pi(a)-\pi(i)}{\pi(j)-\pi(i)}$.
  So we need $a \in (i,j)$ and $\pi(a) \in [\pi(i),\pi(j)]$.
  (Since $\gamma$ is of degree $\pi(j)-\pi(i)$,
  we also need $(\pi(j)-\pi(i))f \in \ZZ$, which as in the Grassmannian case
  we get for free.)
  If these all hold, we get a $\wt T$ weight of
  $$ -\frac{\pi(a)-\pi(i)}{\pi(j)-\pi(i)} y_{\pi(j)} + y_{\pi(a)}
  - \frac{\pi(j)-\pi(a)}{\pi(j)-\pi(i)} y_{\pi(i)} $$

  The same line bundle $\calO(1)$, but now twisted by $\CC_{\pi(i)+\pi(j)-\pi(b)}$,
  has the right isotropy weights to match the
  typical elements in the second column. The $\wt T$-weights of the sections
  of this line bundle, pulled back along $\gamma$, are of the form
  $$ y_{\pi(i)}+y_{\pi(j)}-y_{\pi(b)} + f(-y_{\pi(j)}) + (1-f)(-y_{\pi(i)})
  =  (1-f)y_{\pi(j)} -y_{\pi(b)} + f y_{\pi(i)}
  $$
  for $f\in [0,1]$. For the section to be $\rhocek$-invariant,
  we need  $ (1-f)\pi(j) -\pi(b) + f\pi(i) = 0$,
  hence, $f = \frac{\pi(j)-\pi(b)} {\pi(j)-\pi(i)}$.
  So we need $\pi(b) \in [\pi(i),\pi(j)]$.
  That gives a $\wt T$ weight of
  $$ \qquad\qquad\qquad\qquad\qquad
  \frac{\pi(b)-\pi(i)} {\pi(j)-\pi(i)} y_{\pi(j)}
  - y_{\pi(b)}
  + \frac{\pi(j)-\pi(b)} {\pi(j)-\pi(i)} y_{\pi(i)}
  \qquad\qquad\qquad\qquad\qquad\qedhere
  $$
\end{proof}

This gives an analogue of corollary \ref{cor:Grindices},
that is perhaps simpler:

\begin{Corollary}\label{cor:Flindices}
  Let $(u = u_0 < u_1 < \ldots < u_m = v)$ be a reflection chain in $S_n$,
  and $[\Sigma,\gamma]$ the corresponding point in $X_u^v//_{StMap}\, \rhocek$.
  Then the number of lex-positive weights in the tangent space
  $T_{[\Sigma,\gamma]} X_u^v//_{StMap}\, \rhocek$ is
  \begin{enumerate}
  \item the number of intermediate steps $(u_i r_{ab} < u_i < u_i r_{cd})$,
    $a<b,c<d$ such that $a<c$ or $[a=c$ and $b<d]$,
    plus
  \item for each step $(u_k < u_k r_{ij})$, the number of
    $\{a \in (i,j)\colon u_k(a) \in (u_k(i), u_k(j)) \}$.
  \end{enumerate}
  With these, we can compute $\dim H^{2k}(X_u^v//_{StMap}\, \rhocek;\, \QQ)$
  as the number of reflection chains from $u$ to $v$
  with exactly $k$ lex-positive weights.
\end{Corollary}

\begin{proof}
  The $a<c$ case is exactly as in the Grassmannian case.
  The possibility $a=c$ is new to the flag case, and has us consider the sum
  $$
  \frac{y_a - y_b}{b-a} +
  \frac{- y_a + y_d}{d-a}
  $$
  which is lex-positive exactly if the first summand dominates.
  That is when $b-a < d-a$. 

  The second term is the number of lex-positive weights in the
  displayed formula in lemma \ref{lem:flagsH0}, which is obviously
  exactly half of them.
\end{proof}

It is again straightforward to compute these, a small example being
$X_{32145}^{54132} //_{StMap}\, \rhocek$
with even Betti numbers ${1, 14, 32, 14, 1}$.

% end junk

\junk{
\subsection*{Notes: the general $G/P$ case}

Let $w < w r_\beta$ and consider the $T$-invariant $\PP^1 \subseteq G/P$
connecting $w P/P$, $w r_\beta P/P$. Or for convenience, take $w=1$.
Then we get a curve $\Sigma$ with
$\Sigma^T = \{e \bullet\!\! -\!\!\bullet r_\beta\}$ whose normal
bundle we calculated in the type $A$ case. That should be easily generalized.

Consider the boolean poset $2^{\Delta_1}$ of parabolics $P$, and project
$\pi_P\colon G/B \onto G/P$, giving maps $N_\Sigma G/B \onto N_{\pi(\Sigma)} G/P$.
Since the roots are all different, we should be able to filter
set-theoretically, not just linearly. So there should be a way to
take the 
}

\subsection{A Deodhar-style decomposition}

The Bia\l ynicki-Birula decomposition of $X_u^v//_{StMap}\, \rhocek$
into orbi-cells induces a decomposition of
$(X_u^\circ \cap X^v_\circ)/\rhocek$ closely analogous to 
Deodhar's ``finer decomposition of Bruhat cells'' \cite{Deodhar}
(although that one depends on a choice of reduced word for $v$).

We recall some basics from e.g. \cite[\S 4.1]{Speyer}.
The Deodhar decomposition arises from mixing two cellular decompositions
of the closed Bott-Samelson manifold $BS^Q$: one is into
sub-Bott-Samelsons $\coprod_{R\subseteq Q} BS^R_\circ$,
the other being the Bia\l ynicki-Birula decomposition
$\coprod_{R\subseteq Q} BB(BS^Q)^R_\circ$.
Each decomposition is $B$-invariant, and each cell in either decomposition
contains a unique element of $(BS^Q)^T \iso 2^Q$, but the
sub-Bott-Samelson decomposition arises from an sncd whereas the
BB cells have singular closure and indeed the decomposition
is not even necessarily a stratification. It is obvious a priori that there
are the same number of $k$-cells (specifically, $\#Q \choose k$)
in the two decompositions, because that number can be inferred from
the Betti numbers of $BS^Q$.

The Deodhar decomposition of $BS^Q_\circ$ is then defined by
$$ BS^Q_\circ = \coprod_{R\subseteq Q} BS^Q_\circ \cap BB(BS^Q)^R_\circ $$
(The paper \cite{Deodhar} only addresses this in the case $Q$ reduced,
in which case $BS^Q_\circ$ is isomorphic to a $B$-orbit $X^{\prod Q}_\circ$
in $G/B$.) One of Deodhar's main results is that each of these intersections
is isomorphic to a product of simple factors. If we define
$w_i = \prod_{j\leq i,\, j\in R} r_{Q_j}$ as the partial product of
the simple reflections in $R$, then
$$ BS^Q_\circ \cap BB(BS^Q)^R_\circ \ \iso\
\prod_{i=1}^{\#Q}
\begin{cases}
\begin{tabular}{cc|c|c|}
  & \multicolumn{1}{c}{} & \multicolumn{1}{c}{$i\in R$}  & \multicolumn{1}{c}{$i\notin R$}   \\  \cline{3-4}
  & $w_{i-1}r_{Q_i} > w_{i-1}$ & $pt$ & $\Gm$  \\ \cline{3-4}
  & $w_{i-1}r_{Q_i} < w_{i-1}$ & $\AA^1$ & $\emptyset$  \\ \cline{3-4}
\end{tabular} \\ \\
\end{cases}$$
%  pt &  \text{if $i\in R$ and $w_{i-1} r_{Q_i} > w_{i-1}$} \\
%  \AA^1 & \text{if $i\in R$ and $w_{i-1} r_{Q_i} < w_{i-1}$} \\
%  \Gm & \text{if $i\notin R$ and $w_{i-1} r_{Q_i} > w_{i-1}$} \\
%  \emptyset & \text{if $i\notin R$ and $w_{i-1} r_{Q_i} < w_{i-1}$} 
In particular, one may omit from the decomposition those $R\subseteq Q$
in which one ever meets the $\emptyset$ case. Deodhar calls the $R$ that
avoid this case \defn{distinguished subexpressions}.

One way to derive the notion of ``indistinguished'' is to figure out which
B-B strata $BB(BS^Q)^R_\circ$ are fully contained in some Bott-Samelson divisor
$BS^{Q\setminus q}$ and thus avoid the dense stratum $BS^Q_\circ$.
For $BS^{Q\setminus q}$ to contain $BB(BS^Q)^R_\circ$, it must first contain
the $T$-fixed point, so $q \notin R$. Then, the number of
$\rhocek$-positive weights in $T_R BS^Q$ must match the number of
$\rhocek$-positive weights in $T_R BS^{Q\setminus q}$.
This weight calculation forces $q\in Q$ to be of the $\emptyset$ type
in the product above. Reversing the logic, if there are any such
positions, then there exists a properly contained
$BS^S \subseteq BS^Q$ with $BS^S \supseteq BB(BS^Q)^R_\circ$, hence
that B-B stratum misses $BS^Q_\circ$.

We now perform the same weight calculation for
$M := (X_u^\circ \cap X^v_\circ)/\rhocek$.
Given a reflection chain $\overline w = (u = w_0 < w_1 < \ldots < w_m = v)$,
we have the B-B stratum $M^{\overline w}_\circ \subseteq M$,
whose dimension we computed in corollary \ref{cor:Flindices}.
The divisor strata $D_y \subset M$ are indexed by the open Bruhat
interval $(u,v)$, where $D_y$ contains the point corresponding to
$\overline w$ exactly if $y = w_i$ for some $i \in [1,m)$.
Using theorem \ref{thm:strata} to apply corollary \ref{cor:Flindices} to $D_y$,
we find $\dim (D_y)^{\overline w}_\circ = \dim M^{\overline M}_\circ$ as long
as $w_i$ is a right turn in the sense of corollary \ref{cor:Flindices}(1).

In \cite[\S5.3]{BjBr} are given two formul\ae\ for computing the number
of $\FF_q$-points in an open Richardson variety (the $R$-polynomial):
one is Deodhar's sum over distinguished subexpressions (theorem 5.3.7),
the other is Dyer's sum over certain reflection chains (theorem 5.3.4).
We expect that the decomposition of $(X_u^\circ \cap X^v_\circ)/\rhocek$
induced from the B-B decomposition of $X_u^v//_{StMap}\, \rhocek$
should give a geometric derivation of the latter formula,
and intend to pursue this elsewhere.

\section{Exploiting the $2$-GKMness of Grassmannians}
\label{sec:2GKM}

Let us say that a $T$-variety $M$ is \defn{$d$-GKM} if it has finitely
many $T$-fixed subvarieties of dimension $\leq d$.
(We neglect the very important ``$T$-equivariantly formal'' condition.)
So $0$-GKM means $M^T$ is discrete, and $1$-GKM is the usual
Goresky-Kottwitz-MacPherson condition.  Grassmannians are $2$-GKM,
but most $G/P$ are not (e.g. $GL_3/B$ isn't).  Toric varieties are $d$-GKM for
all $d$. It is easy to see that if $M$ is $d$-GKM, then
$StMap_\beta(\PP^1,M)$ is $(d-1)$-GKM.
The $d$-GKM property is easily seen to be equivalent to the following:
at every $T$-fixed point, each $(d+1)$-tuple of isotropy weights
should be linearly independent. (This ``three-independence'' for
Grassmannians was observed in \cite{GZ}; they observe that the blowup
of a $2$-GKM space at a fixed point is $1$-GKM.)

Using this characterization, it is not hard to show that $2$-GKM flag
manifolds must be cominuscule. However the converse fails:
$SO(5)/P_{short}$ is cominuscule but is not $2$-GKM. (Meanwhile
$SO(5)/P_{long}$ is minuscule and not $2$-GKM.) The simply-laced
cominuscule (and hence minuscule) flag manifolds are indeed $2$-GKM.

The basic thing to compute about a ($1$-)GKM space $M$ is its \defn{GKM graph},
whose vertices and edges are the $T$-fixed points and curves in $M$.
For each such curve $C\iso \PP^1$, we typically want to know the
isotropy weights $wt(T_0 C)$, $wt(T_\infty C)$. This is slightly subtler
in our orbifold context insofar as these points $0,\infty$ may be stackier
than $C$ is generically.

\subsection{Weights on the GKM graph of
  $X_\lambda^\mu//_{StMap}\, \rhocek$}\label{ssec:Grweights}

We do this now for $M = X_\lambda^\mu//_{StMap}\, \rhocek$ the stable map
quotient of a Grassmannian Richardson variety.
The edges correspond to the $T$-fixed curves $C$ in $M$,
which in turn correspond to the $T$-fixed surfaces in $X_\lambda^\mu$.
There are three kinds of such surfaces, and in each case we use
lemma \ref{lem:grH0} to compute the $\wt T$-isotropy weights:
\begin{enumerate}
\item Fix $i,a,j \in [n]$ distinct, and $S \in {[n] \choose k-1}$.
  Then there is a $\PP^2$ worth of $k$-planes $\CC^S \oplus L$
  where $\CC^S$ is the co\"ordinate subspace, and $L \in Gr(1,\CC^{iaj})$.
  The $T$-fixed points in $C \iso Gr(1,\CC^{iaj})//_{StMap}\, \rhocek$
  correspond to the irreducible curve $C_0 := \{ L \leq \CC^{ij} \}$
  and the union $C_\infty := \{ L \leq \CC^{ia} \} \cup \{ L \leq \CC^{aj} \}$.
  Then the $\wt T$-weight of $T_{C_0} C$ is
  $$   - \frac{j-a}{j-i} y_i + y_a - \frac{a-i}{j-i} y_j $$
\item Fix $i,a,j \in [n]$ distinct, and $S \in {[n] \choose k-2}$.
  Then there is a $\PP^2$ worth of $k$-planes $\CC^S \oplus P$
  with $P \in Gr(2,\CC^{iaj})$. 
  The $T$-fixed points in $C \iso Gr(2,\CC^{iaj})//_{StMap}\, \rhocek$
  correspond to the irreducible curve $C_0 := \{ P \geq \CC^{a} \}$
  and the union $C_\infty := \{ P \geq \CC^{i} \} \cup \{ P \geq \CC^{j} \}$.
  Then the $\wt T$-weight of $T_{C_0} C$ is negative that above,
  $$    \frac{j-a}{j-i} y_i - y_a + \frac{a-i}{j-i} y_j $$
\item Fix $a,b,c,d \in [n]$ distinct, and $S \in {[n] \choose k-2}$.
  Then there is a $\PP^1\times\PP^1$ worth of $k$-planes of the form
  $\CC^S \oplus L\oplus M$, where $L \in Gr(1,\CC^{ab}), M \in Gr(1,\CC^{cd})$.
  The $T$-fixed points in
  $C \iso (Gr(1,\CC^{ab})\times Gr(1,\CC^{cd}))//_{StMap}\, \rhocek$
  correspond to the union
  $C_0     := \{ \CC^a \oplus M \} \cup \{ L \oplus \CC^d \}$ and the union
  $C_\infty := \{ L \oplus \CC^c \} \cup \{ \CC^b \oplus M \}$.
  Then the $\wt T$-weight of $T_{C_0} C$ is
  $$     \frac{-y_c + y_d}{d-c} +     \frac{ y_a - y_b}{b-a}     $$
\end{enumerate}

In \S\ref{ssec:Gr} we described the GKM graph vertex set (the $T$-fixed points)
using SRHSTx of shape $\mu\setminus \lambda$. We describe in those
terms now the edges in the GKM graph coming from the above three kinds
of surfaces.  We omit the details of the derivations of these from the above.

\begin{enumerate}
\item Some rim-hook can be sliced along a horizontal edge, or,
  two rim-hooks that abut along one horizontal edge can be
  joined into one rim-hook.
\item Some rim-hook can be sliced along a vertical edge, or,
  two rim-hooks that abut along one vertical edge can be
  joined into one rim-hook.
\item Two rim-hook tableaux agree away from rim-hooks $i,i+1$, but those
  two cannot be joined, because they abut along no edge or along multiple edges.
\end{enumerate}

\section{Some line bundles?}\label{sec:linebundles}

Given a $T$-equivariant line bundle $\calL$ on any space $M$,
the restriction of $c_1(\calL)$ along the map $M^T \into M$
defines a ``moment map''\footnote{Indeed, if one leaves complex {\em algebraic}
  geometry and further endows $\calL$ with a Hermitian connection,
  one can define a more traditional
  moment map $\Phi\colon M \to \RR \tensor_\ZZ T^*$.  We won't make use of this.}
$\Phi\colon M^T \to T^*$, $f \mapsto wt_T(\calL|_f)$. 
This $\Phi$ automatically satisfies the ``GKM conditions''
$wt(T_0 C)\ |\ \left( \Phi(C_\infty) - \Phi(C_0) \right)$,
for each $T$-fixed curve $C \subseteq M$, $C\iso \PP^1$.
For $M$ smooth and proper over $\CC$, Kodaira embedding says
that the line bundle is ample if
$\Phi(C_\infty) - \Phi(C_0)$ is a {\em positive}
multiple of $wt(T_0 C)$ for each $C$.

We were not able to scare up references for some foundational orbifold results
that we expect to be true, concerning the restriction map $K_T(M) \to K_T(M^T)$
in the case $M$ is an orbifold and, say, proper. In the present section,
it would be useful to know that the image of this map is defined by the GKM
conditions (the orbifold analogue of \cite{RosuK,VV}).
In the next section, it would be useful to know that the kernel of this map
is trivial (the closest we could find being \cite[Corollary E]{KR}).

\subsection{An ample line bundle? in the Grassmannian case}

\newcommand\btau{{\bar\tau}}
\begin{Proposition}\label{prop:Phi}
  Fix $\lambda,\mu \in Gr(k,\CC^n)^T$, and index the $T$-fixed points
  in $X_\lambda^\mu//_{StMap} \, \rhocek$ by SRHSTx.
  \junk{
    For $\tau \in (X_\lambda^\mu//_{StMap} \, \rhocek)^T$ thought of as a
    chain $(\nu_j)$ in $[n] \choose k$ from $\lambda$ to $\mu$,
    associate $\beta_j := y_a-y_b$ to the $j$th step where
    $\nu_j = \nu_{j-1}\setminus \{b\} \cup \{a\}$. In rim-hook terms,
    $a,b$ are the diagonals of the SW and NE ends of rim-hook $j$.
    Let $\btau\colon \mu\setminus\lambda \to \NN$ be the associated function
    taking a box in $\mu\setminus\lambda$ to its tableau label.
  }

  Define $\Phi\colon (X_\lambda^\mu//_{StMap}\, \rhocek)^T \to {\wt T}^*$
  as follows. First, standardize the rim-hook tableau $\tau$ by arbitrarily
  breaking ties, then relabel the now totally-ordered cells by
  $1\ldots \#(\mu\setminus\lambda)$. Now fix the arbitrariness by
  relabeling each rim-hook by
  the average of the numbers in the rim-hook. Let $\Psi(\tau)$ denote
  this rational-valued tableau.

  Finally, let $\Phi(\tau)$ be the sum
  $ \sum_{i=1}^{n-1} \big(\text{total of the labels in diagonal $i$ in }
    \Psi(\tau)\big)
  \, (y_i-y_{i+1})$.

  Then $ \Phi(C_\infty) - \Phi(C_0)$ is a positive multiple of $wt(T_0 C)$
  for each $T$-fixed curve $C$, and in particular satisfies the GKM conditions.
\end{Proposition}

% end junk

For example in $Gr(3,\CC^6)$,
$$ \tau =  \raisebox{-.5cm}{\begin{tikzpicture}[scale=0.5]
\draw (0,1) -- (0,3) -- (1,3) -- (1,4) -- (2,4) -- (3,4);
\path (0.5,2.5) node {1} ;
\draw (0,2) -- (1,2) -- (1,3) -- (1,4) -- (2,4) -- (3,4);
\path (1.5,3.5) node {2} ;
\draw (0,2) -- (1,2) -- (2,2) -- (2,3) -- (2,4) -- (3,4);
\path (2.5,3.5) node {3} ;
\draw (0,2) -- (1,2) -- (2,2) -- (2,3) -- (3,3) -- (3,4);
\end{tikzpicture}}
\ \xmapsto{\text{break ties}}\
 \raisebox{-.5cm}{\begin{tikzpicture}[scale=0.5]
\draw (0,1) -- (0,3) -- (1,3) -- (1,4) -- (2,4) -- (3,4);
\path (0.5,2.5) node {1} ;
\draw (0,2) -- (1,2) -- (1,3) -- (1,4) -- (2,4) -- (3,4);
\path (1.5,3.5) node {2'} ;
\path (1.5,2.5) node {2} ;
\draw (0,2) -- (1,2) -- (2,2) -- (2,3) -- (2,4) -- (3,4);
\path (2.5,3.5) node {3} ;
\draw (0,2) -- (1,2) -- (2,2) -- (2,3) -- (3,3) -- (3,4);
\end{tikzpicture}}
\ \xmapsto{\text{standardize}}\
\raisebox{-.5cm}{\begin{tikzpicture}[scale=0.5]
\draw (0,1) -- (0,3) -- (1,3) -- (1,4) -- (2,4) -- (3,4);
\path (0.5,2.5) node {1} ;
\draw (0,2) -- (1,2) -- (1,3) -- (1,4) -- (2,4) -- (3,4);
\path (1.5,3.5) node {3} ;
\path (1.5,2.5) node {2} ;
\draw (0,2) -- (1,2) -- (2,2) -- (2,3) -- (2,4) -- (3,4);
\path (2.5,3.5) node {4} ;
\draw (0,2) -- (1,2) -- (2,2) -- (2,3) -- (3,3) -- (3,4);
\end{tikzpicture}}
\ \xmapsto{\text{average within each rim-hook}} \
\raisebox{-.5cm}{\begin{tikzpicture}[scale=0.5]
\draw (0,1) -- (0,3) -- (1,3) -- (1,4) -- (2,4) -- (3,4);
\path (0.5,2.5) node {1} ;
\draw (0,2) -- (1,2) -- (1,3) -- (1,4) -- (2,4) -- (3,4);
\path (1.5,3.5) node {\tiny{$2\frac 1 2$}} ;
\path (1.5,2.5) node {\tiny{$2\frac 1 2$}} ;
\draw (0,2) -- (1,2) -- (2,2) -- (2,3) -- (2,4) -- (3,4);
\path (2.5,3.5) node {4} ;
\draw (0,2) -- (1,2) -- (2,2) -- (2,3) -- (3,3) -- (3,4);
\end{tikzpicture}}
$$
summed along diagonals gives
$0(y_1-y_2) + 1(y_2-y_3) + 2\frac{1}{2}((y_3-y_4)+(y_4-y_5)) + 4(y_5-y_6)$.

If the orbifold analogue of \cite{RosuK,VV} holds, then $\Phi$ is
associated to an equivariant line bundle. Until such time as this
result is in the literature, we claim only the combinatorial statement.

\begin{proof}
  We have three kinds of $T$-fixed curves $C$ to check, each connecting
  two $T$-fixed points corresponding to SRHSTx $\tau,\tau'$.
  We refer to \S \ref{ssec:Grweights} for the weights $wt(T_0 C)$.
%  Let $H_i$, $H'_i$ denote the set of boxes in rim-hook $i$ in $\tau$,
%  $\tau'$ respectively.

  \begin{enumerate}
  \item Let $\tau'$ be $\tau$ with rim-hook $H$ cut along a horizontal edge,
    into lower and upper rim-hooks $H_1, H_2$. Then $\Psi(\tau)$
    can be derived from $\Psi(\tau')$ by averaging $H_1,H_2$ together.
    In particular, $\Psi(\tau')-\Psi(\tau)$ is zero outside $H$,
    constant on $H_1$, $H_2$, and has total $0$. If $H_1$ has boxes in
    diagonals $[i,a-1]$ and $H_2$ in boxes $[a,j-1]$, and before averaging
    the smallest value in $H_1$ is $m+1$, then the averages in
    $H_1$, $H_2$, $H_1 \cup H_2$ are $m + \frac{a-i}{2}$,
    $m + a-i + \frac{j-a}{2}$, $m + \frac{j-i}{2}$ respectively.
    Hence the average in $H_1\cup H_2$, minus the averages in $H_1$ or $H_2$,
    are $\frac{j-a}{2}$, $\frac{i-a}{2}$ respectively. From here we compute
    $\Phi(\tau')-\Phi(\tau) = \frac{1}{2}((j-a)y_i - (j-i)y_a + (a-i)y_j)$,
    which is $\frac{j-i}{2}$ times the weight from (1)
    in \S\ref{ssec:Grweights}.
  \item Cutting along a vertical edge gives nearly the same calculation,
    resulting in $\frac{j-i}{2}$ times the weight from (2)
    in \S\ref{ssec:Grweights}.
  \item Let $H_1,H_2$ be two rim-hooks in $\tau$ with adjacent values $i,i+1$,
    where $H_1$ meets diagonals $[a,b-1]$ and $H_2$ meets diagonals $[c,d-1]$.
    If $m+1$ is the smallest value in $H_1$ before averaging, then the
    averages in $H_1,H_2$ are $m + \frac{b-a}2$, $m+b-a+\frac{d-c}2$
    respectively. After swapping $i \leftrightarrow i+1$ (except in
    the diagonals containing both) those averages in $H_2,H_1$ are now
    $m + \frac{d-c}2$, $m+d-c+\frac{b-a}2$ respectively. Hence the
    change in the diagonal sum is $d-c$ along $[a,b-1] \setminus [c,d-1]$,
    $a-b$ along $[c,d-1] \setminus [a,b-1]$. 
    From here we compute
    $\Phi(\tau')-\Phi(\tau) = (b-a)(-y_c+y_d) + (d-c)(y_a-y_b)$,
    which is $(d-c)(b-a)$ times the weight from (3) in \S\ref{ssec:Grweights}.
    \qedhere
  \end{enumerate}
  
\end{proof}

% \subsection{The anticanonical bundle, in the Grassmannian case}

\subsection{An ample line bundle? for the flag manifold case}

Fix a Richardson variety $X_u^v \subseteq Fl(n)$.
Let $\pi_k\colon Fl(n) \to Gr(k,\CC^n)$ be the projection,
and define $\lambda,\mu \in Gr(k,\CC^n)^T$ by $\pi_k(X_u^v) = X_\lambda^\mu$. 
Write $\pi'_k$ for the induced map $\pi_k//_{StMap}\, \rhocek$
on stable map quotients (see \S\ref{ssec:maps}). We get a square
$$
\begin{matrix}
  (X_u^v//_{StMap}\, \rhocek)^T &\into& X_u^v//_{StMap}\, \rhocek \\
  \downarrow\pi'_k &&   \downarrow\pi'_k \\
  (X_\lambda^\mu//_{StMap}\, \rhocek)^T &\into& X_\lambda^\mu//_{StMap}\, \rhocek
\end{matrix}
$$

If we knew that proposition \ref{prop:Phi}'s check of the GKM conditions
implied that $\Phi$ comes from a line bundle $\calL$, we'd be able to
pull back $\calL$ to obtain a line bundle $\calL_k$
on $X_u^v//_{StMap}\, \rhocek$.
Until such a result is available, we pull back along the fixed points
to obtain a class on $(X_u^v//_{StMap}\, \rhocek)^T$ satisfying the
Chang-Skjelbred conditions. Then we sum those classes over $k$,
corresponding to taking the tensor product $\tensor_k \calL_k$.

\section{Anticanonicality and the Bj\"orner-Wachs
  boundary}\label{sec:anticanonical}

\subsection{A criterion for anticanonicality of a divisor}

Let $D \subseteq M$ be a simple normal crossings divisor on an orbifold,
$D$ having components $(D_i)_{i\in I}$, invariant
under a $T$-action on $M$ with isolated fixed points.
Let $\wt T$ be a connected extension of $T$ that acts on $M$'s tangent spaces.
Write $D_J := \cap_{j\in J} D_j$ for the sncd stratum in $D$.
Consider four conditions:
\begin{enumerate}
\item $D$ is \defn{equivariantly anticanonical}, meaning it is given by the
  vanishing of a $\wt T$-{\em invariant} section of the anticanonical bundle.
\item Within each stratum $D_J$, the divisor $\Union_{i\notin J} D_{J \cap \{i\}}$
  is equivariantly anticanonical.
\item For each $f\in M^T$, the restriction $[D \subseteq M]|_f$
  of the divisor class is the sum of the $\wt T$-weights on $T_f M$.
\item For each $f\in M^T$, and $D_J$ the minimal stratum containing $f$ (that is,
  $J = \{j\colon D_j\ni f\}$), the sum of the $\wt T$-weights on $T_f D_J$ vanishes.  
\end{enumerate}

Condition (4) holds trivially when the $T$-fixed point $f$ is a stratum
(i.e. $D_J = \{f\}$).
For example, if $M$ is a smooth toric variety and $D$ the complement
of the torus, then every $T$-fixed point is a stratum,
making (4) easy to check.

\begin{Lemma}\label{lem:anticanonical}
  $
  \begin{matrix}
    (1) &\Longleftrightarrow& (2) \\
    \rotatebox[origin=c]{270}{$\implies$} && \rotatebox[origin=c]{270}{$\implies$} \\
    (3) &\Longleftrightarrow & (4)
  \end{matrix}
  $
%  (1) $\implies$ (2),(3). (2) $\implies$ (4). (3) $\Longleftrightarrow$ (4).

  If $M$ is a smooth projective variety (not stacky), then these are
  all equivalent. More generally, they are equivalent when GKM theory
  holds for the restriction map $K_T(M) \to K_T(M^T)$. 
\end{Lemma}

\begin{proof}
  (1) is the trivial case $D_\emptyset$ of (2),
  and implies (2) by an inductive application of the adjunction formula.

  (1) implies (3) because we can restrict the anticanonical class to
  a neighborhood of $f$, where the tangent bundle becomes a sum of line bundles.
  Then the anticanonical class is the sum of the weights of those
  line bundles. (2) implies (4) by the same argument.

  For $f \in (D_J)^T$, we have a $\wt T$-equivariant isomorphism
  $$ T_f M  \iso T_f D_J \oplus \Oplus_{j\in J} (T_f M / T_f D_j)$$
  Hence if $A$ is the anticanonical class,
  $$ A|_f = \sum \{\text{weights of }T_f D_J\} +
  \sum_J wt_{\wt T}(T_f M / T_f D_j) $$
  Meanwhile, by definition of $J$,
  $$ [D\subseteq M]|_f = \sum_J [D_j\subseteq M]|_f
  = \sum_J wt_{\wt T}(T_f M / T_f D_j) $$
  So $[D\subseteq M]|_f = A|_f$ (condition (3)) iff
  $\sum \{\text{weights of }T_f D_J\} = 0$ (condition (4)).
  
  Finally, if $M$ is a smooth projective variety, then the restriction to
  fixed points is injective, making the (1)$\implies$(3), (2)$\implies$(4)
  implications reversible.
\end{proof}

\junk{
  Again, were someone to establish that
  the map $K_T(M) \to K_T(M^T)$ is injective for $T$-actions on
  smooth orbifolds with (say) equivariantly projective coarse moduli space,
  then all four conditions in lemma \ref{lem:anticanonical}
  would be found to be equivalent.
  } % end junk

\subsection{Checking the criterion on $X_u^v//_{StMap}\, \rhocek$}

We confirm (4) for the case $M = X_u^v//_{StMap}\, \rhocek$ where
$X_u^v \subseteq Fl(n)$. The components $(D_w)$ are indexed by $w \in (u,v)$
the open interval, where a general point in $D_w$ consists of
curves $\Sigma$ with one node, mapping to $w$.  A $T$-fixed point $f$
is specified by a reflection chain $(w_i)_{i=0}^m$, lying in the
stratum $\bigcap_{i=1}^{m-1} D_{w_i}$.

Assume to begin with that $f \notin D$, i.e., the reflection chain is
simply $(u,v)$ itself (not always possible, of course; this requires
$u^{-1}v$ to be a reflection).
Then adding up the tangent weights from lemma \ref{lem:flagsH0},
we get $0$, confirming (4).

In the general case, by theorem \ref{thm:strata}(1) the stratum
$\bigcap_{i=1}^{m-1} D_{w_i}$ is a product of individual
$X_{w_i}^{w_{i+1}}//_{StMap}\, \rhocek$.
Using this we reduce to the case of one factor, as handled just above.

\subsection{The Grassmannian case}

In \cite{KollarXu} one finds the ``folklore conjecture'' that
the dual simplicial complex to an anticanonical sncd should be closely
related to a sphere. However, the dual simplicial complex to the
boundary of $X_\lambda^\mu//_{StMap}\, \rhocek$,
for $X_\lambda^\mu \subseteq Gr(k,\CC^n)$, is typically a ball not
a sphere (theorem \ref{thm:BjornerWachs}); is this divisor not anticanonical?

\begin{Proposition}
  Let $\Delta$ be the order complex of the open interval $(\lambda,\mu)$
  in Grassmannian Bruhat order, and $F \in \Delta$ a \defn{ridge}
  (codimension $1$ face). So by theorem \ref{thm:strata},
  $F$ corresponds dually to a one-dimensional
  stratum $C$ in $X_\lambda^\mu//_{StMap}\, \rhocek$.

  If $F$ is an \defn{interior} ridge, contained in two facets, then
  $C$ has trivial orbifold stabilizers. If $F$ is an \defn{exterior} ridge,
  contained in one facet, then one $T$-fixed point on $C$ is a stratum,
  and the other is a $[pt/Z_2]$ which is not a stratum.
\end{Proposition}

\begin{proof}
  Using theorem \ref{thm:strata}(1) we reduce to the case
  $C = X_\lambda^\mu//_{StMap}\, \rhocek$ where $\#(\mu\setminus \lambda)=2$.
  The skew partition $\mu\setminus \lambda$ can either be a horizontal
  domino, a vertical domino, or two separated squares (closely related
  to the case analysis in \S\ref{ssec:Grweights}).
  The corresponding Richardson surfaces are $Gr(1,\CC^{i,i+1,i+2})$,
  $Gr(2,\CC^{i,i+1,i+2})$, $Gr(1,\CC^{i,i+1})\times Gr(1,\CC^{j,j+1})$
  respectively.

  On each of those, the $\rhocek$-stabilizers are generically trivial.
  In the first two, one finds $T$-fixed curves $Gr(1, \CC^{i,i+2})$,
  $Gr(1,\CC^{i,i+1,i+2}/\CC^{i+1})$ respectively whose $\rhocek$-stabilizers
  are $Z_2$, but each of the four $T$-fixed curves in the third case
  has trivial $\rhocek$-stabilizer.
\end{proof}

This hints at, in the orbifold situation, a modification to the
folklore conjecture: the dual simplicial complex should be closely
related to a sphere or ball, whose boundary faces (the union of the
exterior ridges) correspond to strata containing orbifold points.
We have not explored the landscape of anticanonical sncds on orbifolds enough
to make a more precise conjecture.

\bibliographystyle{alpha}    % it seems this does nothing.

\begin{thebibliography}{10}

  
\bibitem[Ale06]{Alexeev} Valery Alexeev,
  Higher-dimensional analogues of stable curves. \\
  Proceedings of the International Congress of Mathematicians, \\
  Madrid, August 22--30, 2006: invited lectures.
  \url{https://arxiv.org/abs/math/0607682}

\bibitem[ArKhLatParRa]{Adeel2}
  Dhyan Aranha, Adeel A. Khan, Alexei Latyntsev, Hyeonjun Park,
  Charanya Ravi, \\ Virtual localization revisited.
  Preprint 2022. \url{https://arxiv.org/abs/2207.01652}

\bibitem[BjBre05]{BjBr} Anders Bj\"orner, Francesco Brenti,
  Combinatorics of Coxeter groups. \\ Graduate Texts in Mathematics vol. 231.
  New York: Springer, 2005.
  
\bibitem[BjWa82]{BjornerWachs} Anders Bj\"orner, Michelle Wachs,
  Bruhat order of Coxeter groups and shellability. \\
  Advances in Mathematics 43.1 (1982): 87--100.
  
\bibitem[Bri05]{Brion} Michel Brion,
  Lectures on the geometry of flag varieties.
  In Topics in cohomological studies of algebraic varieties,
  Trends Math., pages 33--85. Birkh\"auser, Basel, 2005.

\bibitem[ChSa13]{ChenSatriano} Qile Chen, Matthew Satriano,
  Chow quotients of toric varieties as moduli of stable log maps.
  Algebra \& Number Theory 7.9 (2013): 2313--2329.
  \url{https://arxiv.org/abs/1201.3406}

\bibitem[CoDesNSuWe22]{WebbEtAl} Elsa Corniani, Neeraj Deshmukh, Brett Nasserden,
  Emanuel Reinecke, Nawaz Sultani, Rachel Webb,
  The stack of admissible covers is algebraic. \\
  Stacks Project Expository Collection,
  Cambridge University Press, 2022.

\bibitem[Deo85]{Deodhar} Vinay V. Deodhar,
  On some geometric aspects of Bruhat orderings. I. A finer
  decomposition of Bruhat cells. Inventiones mathematicae 79.3
  (1985): 499--511.

\bibitem[E16]{Escobar} Laura Escobar,
  Brick manifolds and toric varieties of brick polytopes. \\
  The Electronic Journal of Combinatorics 23.2 (2016): P2-25.
  \url{https://arxiv.org/abs/1404.4671}

\bibitem[F24]{FerroniEtAl} Luis Ferroni, Jacob P. Matherne, Lorenzo Vecchi,
  Chow functions for partially ordered sets. \\
  Preprint 2024.
  \url{https://arxiv.org/abs/2411.04070}
  
\bibitem[FuPan97]{FP} William Fulton, Rahul Pandharipande,
  Notes on stable maps and quantum cohomology. \\
  {\it Algebraic geometry---Santa Cruz 1995}, 45--96,
  Proc. Sympos. Pure Math., 62, Part 2. \\
  \url{https://arxiv.org/abs/alg-geom/9608011}
  
\bibitem[GZa01]{GZ} Victor Guillemin and Catalin Zara.
  $1$-skeleta, Betti numbers, and equivariant cohomology. \\
  Duke Math. J., 107(2):283--349, 2001.
  \url{https://arxiv.org/abs/math/9903051}

\bibitem[Ka93]{Kapranov} M. M. Kapranov,
  Chow quotients of Grassmannians. I, Adv. Soviet Math., vol. 16, \\
  Amer. Math. Soc., Providence, RI, 1993.
  \url{https://arxiv.org/abs/alg-geom/9210002}

\bibitem[KaStZe91]{KSZ} \bysame, B. Sturmfels, A. V. Zelevinsky,
  Quotients of toric varieties. \\
  Math. Annalen (1991), Volume 290, pages 643--655.

\bibitem[KhRa]{KR} Adeel A. Khan, Charanya Ravi,
  Cohomological and categorical concentration. Preprint 2023. \\
  \url{https://arxiv.org/abs/2308.01652}

\bibitem[KnLamSp14]{projRich} Allen Knutson, Thomas Lam, David E Speyer,
  Projections of Richardson varieties. \\
  Journal f\"ur die reine und angewandte Mathematik (Crelle's Journal)
  2014.687 (2014): 133--157. \\
  \url{https://arxiv.org/abs/1008.3939}

\bibitem[Kn20]{Allensdottirs} \bysame,
  Stable map resolutions of Richardson varieties.
  Allensdottirs seminar September 2020.
  \url{https://pi.math.cornell.edu/~allenk/agctalk.pdf}

\bibitem[KoX16]{KollarXu} J\'anos Koll\'ar, Chenyang Xu,
  The dual complex of Calabi--Yau pairs. \\
  Invent. math. 205, 527--557 (2016). 
  \url{https://arxiv.org/abs/1503.08320}

\bibitem[Kw07]{Kwon} Seongchun Kwon,
  Real aspects of the moduli space of genus zero stable maps.
  Vol 31, Issue 3 (2007), Turkish Journal of Mathematics, p303--317.
  \url{https://arxiv.org/abs/math/0305128v4}
  
\bibitem[Pay13]{Payne} Sam Payne,
  Boundary complexes and weight filtrations. \\
  Mich. Math. J. 62 (2013), no. 2, 293--322.
  \url{https://arxiv.org/abs/1109.4286}

\bibitem[Ro03]{RosuK} Ioanid Ro\c su,
  Appendix to: Equivariant $K$-theory and equivariant cohomology. \\
  Math Z 243, 423--448 (2003). 
  \url{https://arxiv.org/abs/math/9912088}
  
\bibitem[Sp25]{Speyer} David E Speyer,
  Richardson varieties, projected Richardson varieties and positroid varieties.\\
  To appear in the Handbook of Combinatorial Algebraic Geometry, 2025.\\
  \url{https://arxiv.org/abs/2303.04831}  
  
\bibitem[TY09]{TY} Hugh Thomas, Alexander Yong,
  A combinatorial rule for (co)minuscule Schubert calculus. \\
  Advances in mathematics 222.2 (2009): 596--620.  
  \url{https://arxiv.org/abs/math/0608276}

\bibitem[VeVi03]{VV} Gabriele Vezzosi, Angelo Vistoli,
  Higher algebraic $K$-theory for actions of diagonalizable groups.
  Invent. math. 153, 1--44 (2003) 
  \url{https://arxiv.org/abs/math/0107174}
  
\end{thebibliography}

\end{document}